\documentclass[10pt,a4paper]{article}
\usepackage{fullpage,mathrsfs,graphicx,framed,color,amssymb,amsmath,amsthm}
\usepackage[boxed, noline, ruled, linesnumbered]{algorithm2e}
\usepackage{epsfig, graphicx}
\usepackage{latexsym,amsfonts,amsbsy,amssymb}
\usepackage{amsmath,amsthm}
\usepackage{enumerate}

\makeatletter %'@' is now a normal "letter" for TeX

\@addtoreset{equation}{section}
\makeatother %'@' is restored as a "non-letter" character for TeX
\allowdisplaybreaks % For long formula
\newtheorem{theorem}{Theorem}[section]
\newtheorem{lemma}{Lemma}[section]
\newtheorem{corollary}{Corollary}[section]
\newtheorem{remark}{Remark}[section]
\begin{document}
\title{Enhanced Error Estimates for Augmented Subspace Method with Crouzeix-Raviart Element}
\author{Zhijin Guan\footnote{LSEC, ICMSEC,
Academy of Mathematics and Systems Science, Chinese Academy of
Sciences,  Beijing 100190, China, and School of Mathematical
Sciences, University of Chinese Academy of Sciences, Beijing 100049, China.
({\tt guanzhijin@lsec.cc.ac.cn})},\ \ \
Yifan Wang\footnote{LSEC, ICMSEC,
Academy of Mathematics and Systems Science, Chinese Academy of
Sciences,  Beijing 100190, China, and School of Mathematical
Sciences, University of Chinese Academy of Sciences, Beijing 100049, China. ({\tt wangyifan@lsec.cc.ac.cn})},\ \ \ Hehu Xie\footnote{LSEC, ICMSEC,
Academy of Mathematics and Systems Science, Chinese Academy of
Sciences,  Beijing 100190, China, and School of Mathematical
Sciences, University of Chinese Academy of Sciences, Beijing 100049, China. ({\tt hhxie@lsec.cc.ac.cn})} \  \  and\  \ Chenguang Zhou\footnote{Department of Mathematics, Beijing University of Technology, Beijing 100124, China. (Corresponding author: {\tt zhoucg@bjut.edu.cn})} }
\date{}
\maketitle
\begin{abstract}
    In this paper, we present some enhanced error estimates for augmented subspace methods with the nonconforming Crouzeix-Raviart (CR) element. Before the novel estimates, we derive the explicit error estimates for the case of single eigenpair and multiple eigenpairs based on our defined spectral projection operators, respectively. Then we first strictly prove that the CR element based augmented subspace method exhibits the second-order convergence rate between the two steps of the augmented subspace iteration, which coincides with the practical experimental results. The algebraic error estimates of second order for the augmented subspace method explicitly elucidate the dependence of the convergence rate of the algebraic error on the coarse space, which provides new insights into the performance of the augmented subspace method. Numerical experiments are finally supplied to verify these new estimate results and the efficiency of our algorithms.


\end{abstract}

\section{Introduction}
Solving large-scale eigenvalue problems is a basic and challenging task in the field of scientific and engineering computing. Compared with linear boundary value problems, it is always more difficult to solve eigenvalue problems because it requires more computational work and memory. In order to solve these large-scale sparse eigenvalue problems, Krylov subspace type methods (Implicitly Restarted Lanczos/Arnoldi Method (IRLM/IRAM) \cite{Sorensen}), Preconditioned INVerse ITeration (PINVIT) method \cite{BramblePasciakKnyazev,PINVIT,Knyazev}, Locally Optimal Block Preconditioned Conjugate Gradient (LOBPCG) method \cite{Knyazev_Lobpcg,KnyazevNeymeyr}, Jacobi-Davidson-type techniques \cite{Bai} and Generalized Conjugate Gradient Eigensolver (GCGE) \cite{LiXieXuYouZhang,ZhangLiXieXuYou} have been developed.
All these popular methods include the orthogonalization steps during computing Rayleigh-Ritz problems
which are always the bottlenecks for designing efficient parallel schemes
for determining relatively many eigenpairs.

As one of the effective methods to solve eigenvalue problems,
the two-grid method has been proposed and analyzed in \cite{XuZhou} for the linear eigenvalue problem. This algorithm requires solving a small-scale eigenvalue problem on a coarse mesh and a large-scale linear boundary value problem on a fine mesh. When the mesh sizes of the coarse mesh ($H$) and the fine mesh ($h$) have an appropriate proportional relation ($H=\sqrt{h}$), the optimal error estimate for the approximate solution can be derived. However, owing to the strict constraints on the ratio (i.e., $H=\sqrt{h}$), the two-grid method only performs on two mesh levels and can not be used to design the eigensolver for the algebraic eigenvalue problems which come from the finite element discretization of the eigenvalue problems of the differential operators.

Recently, a type of augmented subspace methods and their multilevel correction methods is proposed for solving eigenvalue problems in \cite{full,HongXieXu,LinXie_MultiLevel,Xie_IMA,Xie_JCP,XieZhangOwhadi,XuXieZhang}.
In this type of methods, there exists an augmented subspace which is used in each correction step and constructed with the low dimensional finite element space defined on the coarse grid. The idea of augmented subspace gives birth to the type of
augmented subspace methods which only needs the low dimension finite element space on the coarse mesh and the final finite element space on the finest mesh.
The augmented subspace methods can transform
the solution of the eigenvalue problem on the final level of mesh to the
solution of linear boundary value problems on the final level of mesh and the solution
of the eigenvalue problem on the low dimensional augmented subspace.
This type of methods can also work even the coarse and finest meshes have no
nested property \cite{Dang}.
Based on the augmented subspace methods, the multilevel correction methods give the
ways to construct the multigrid methods for eigenvalue problems \cite{full,HongXieXu,Xie_IMA,Xie_JCP,XieZhangOwhadi}, and also in \cite{XuXieZhang}, the authors design an eigenpair-wise parallel eigensolver
for the eigenvalue problems. This type of parallel method avoids doing orthogonalization and inner-products in the high dimensional space which accounts for a large portion of the wall time in the parallel computation.
The above references are mainly investigated based on conforming finite element methods, and in \cite{MR3434038, MR3395398}, the authors extend this type of methods to the case of the nonconforming finite element methods for the Laplace eigenvalue problem and the Steklov eigenvalue problem, respectively.

In \cite{MR3434038}, the author first illustrates the error estimate of the solution operator with respect to the eigenspace in Theorem 2.1. Then the augmented subspace method and its multilevel correction method is presented based on the nonconforming finite element method, and the algebraic error estimates are evaluated by utilizing the finest mesh size in Corollary 5.2, which is called ``superclose" in Remark 5.3. However, these derived algebraic error estimates only prove the first order convergence rate, while the second order algebraic error accuracy is obtained in the numerical experiments.

Based on above, in this paper, we provide some new and enhanced error estimates from the following two aspects:
\begin{itemize}
\item In the spatial discretization of CR element, we separately derive the explicit error estimates for the case of single eigenpair and multiple eigenpairs based on our defined spectral projection operators. These estimates depict the relationship between the errors of spectral projections of the eigenvalue problem and the errors of finite element projection of the corresponding linear boundary value problem with an explicit constant expression related to the eigenvalues and their gap.
\item For the first time, we strictly prove that the augmented subspace method based on the nonconforming CR element exhibits the second-order convergence rate between the two augmented subspace iteration steps, which coincides with the practical experimental results.
\end{itemize}
Our proposed algebraic error estimates are superior to the results in \cite{MR3434038}. More importantly, these algebraic error estimates of second order for the augmented subspace method explicitly elucidate the dependence of the convergence rate on the coarse space, which provides new insights into the performance of the augmented subspace method.

The overall structure of this paper goes as follows. In Section \ref{Section_2}, we introduce the nonconforming CR element method for the eigenvalue problem and derive new error estimates based on our defined spectral projection operators. The augmented subspace method and the algebraic error estimates of second order will be given in Section \ref{Section_3}. In Section \ref{Section_4}, numerical experiments are presented to validate the theoretical error estimates and the efficiency of our algorithms.
Some concluding remarks are provided in the last section.

%============================================================
\section{CR element discretization for the eigenvalue problem}\label{Section_2}

In our methodology description, we set the Laplace eigenvalue problem as an example. The framework of the theoretical analysis of the method can also be developed to other eigenvalue problems, for example, linear elasticity eigenvalue problems. Additionally, it should be noted that the letter $C$ (with or without subscripts) denotes a generic positive constant which may be different at its different occurrences throughout this paper.

Now, let us concern with the following Laplace eigenvalue problem: Find $(\lambda,u)\in\mathbb R\times H_0^1(\Omega)$, such that
\begin{eqnarray}\label{originproblem}
\left\{
\begin{array}{rcl}
-\Delta u&=& \lambda u,\ \ \ {\rm in}\ \Omega,\\
u&=&0,\ \ \ \ {\rm on}\ \partial\Omega,\\
|u|_1^2&=&1,
\end{array}
\right.
\end{eqnarray}
where $|\cdot|_1$ represents $H^1$-type semi-norm (cf. \cite{MR0450957}) and $\Omega\subset\mathbb R^2$ is a bounded and convex domain with Lipschitz boundary $\partial\Omega$. Let $V=H_0^1(\Omega)$ and $W=L^2(\Omega)$. Then the variational formulation of \eqref{originproblem} is provided as: Find $(\lambda,u)\in\mathbb R\times V$ such that $a(u,u)=1$ and
\begin{eqnarray}\label{weak_eigenvalue_problem}
a(u,v)=\lambda b(u,v),\quad \forall v\in V,
\end{eqnarray}
where $a(\cdot, \cdot)$ and $b(\cdot,\cdot)$ are defined as follows
\begin{eqnarray*}
a(u,v)=\int_{\Omega}\nabla u\cdot\nabla vd\Omega, &&
b(u,v)=\int_{\Omega}uvd\Omega.
\end{eqnarray*}
Based on the bilinear forms $a(\cdot,\cdot)$ and $b(\cdot,\cdot)$, we can respectively define the
norms on the spaces $V$ and $W$ as
\begin{eqnarray}\label{Nomr_a}
\left\|v\right\|_a = \sqrt{a(v,v)},\ \ \ \ \forall v\in V,
\end{eqnarray}
and
\begin{eqnarray}\label{Nomr_b}
\left\|w\right\|_b = \sqrt{b(w,w)},\ \ \ \ \forall w\in W.
\end{eqnarray}
We can find that the norms $\|\cdot\|_a$ and $\|\cdot\|_b$ defined in \eqref{Nomr_a} and \eqref{Nomr_b} are equivalent to the $H^1$-type semi-norm $|\cdot|_1$ and $L^2$ norm $\|\cdot\|_0$, respectively. It is well known that the eigenvalue problem (\ref{weak_eigenvalue_problem})
has an eigenvalue sequence $\{\lambda_j \}$ (cf. \cite{BabuskaOsborn_1989,Chatelin})
$$0<\lambda_1\leq \lambda_2\leq\cdots\leq \lambda_k\leq \cdots,\ \ \ \lim_{k\rightarrow\infty}\lambda_k=\infty.$$
And the associated eigenfunctions are provided as
$$u_1, u_2, \cdots, u_k, \cdots.$$
Here $a(u_i,u_j)=\delta_{ij}$ ($\delta_{ij}$ denotes the Kronecker function).

Let $\mathcal{T}_{h}$ be a quasi-uniform triangulation of $\Omega$. Denote by $\mathcal{E}_h$ the set of all edges of $\mathcal{T}_h$. $\mathcal{E}_h=\mathcal{E}_h^i\cup \mathcal{E}_h^b$, where $\mathcal{E}_h^i$ denotes the interior edge set and $\mathcal{E}_h^b$ denotes the edge set lying on the boundary $\partial\Omega$. The CR element space $V_h$ is defined as
\begin{align}\label{CR_Definition}
V_h:= &\Big\{v\in L^2(\Omega):v|_{K}\in
{\rm span}\{1,x,y\},\int_{\ell} v|_{K_1}{\rm d}s=\int_{\ell} v|_{K_2}{\rm d}s,\nonumber\\
&\hspace{0cm}\text{ when }K_1\cap K_2=\ell\in\mathcal{E}_h^i\
\text{and }\int_{\ell}v|_{K}{\rm d}s=0,  \text{ if }\ \ell\in\mathcal{E}_h^b\Big\},
\end{align}
where $K,\ K_1,\ K_2\in \mathcal{T}_h$.

The CR element approximation to \eqref{weak_eigenvalue_problem} is defined as follows: Find $(\bar\lambda_h, \bar u_h) \in \mathbb{R}\times V_h$ such that $a_h(\bar u_h,\bar u_h)=1$ and
\begin{eqnarray}\label{diseigenproblem}
a_h(\bar u_h,v_h)=\bar\lambda_h b(\bar u_h,v_h), \quad \forall v_h\in V_h,
\end{eqnarray}
where $a_h(\cdot,\cdot)$ is defined as
\begin{eqnarray*}
a_h(w_h,v_h)=\sum\limits_{K\in\mathcal{T}_h}\int_K\nabla w_h\nabla v_h{\rm d}K,\quad \forall w_h,v_h\in V_h.
\end{eqnarray*}

The bilinear form $a_h(\cdot,\cdot)$ is $V_h$-elliptic on $V+V_h$. Hence, we define the norms $\|\cdot\|_{a,h}$ and $\|\cdot\|_b$ on $V+V_h$ as
\begin{eqnarray}\label{ahb}
\|v\|_{a,h}^2=a_h(v,v),\ \ \|v\|_b^2=b(v,v), \ \ \ \text{for}\ v\in V+V_h.
\end{eqnarray}

For the eigenvalue problem \eqref{diseigenproblem}, the Rayleigh quotient holds for the eigenvalue $\bar\lambda_h$,
\begin{eqnarray*}\label{Dis_Rayleigh_quotient}
\bar\lambda_h=\frac{a_h(\bar u_h,\bar u_h)}{b(\bar u_h,\bar u_h)}.
\end{eqnarray*}

Similarly, the discrete eigenvalue problem \eqref{diseigenproblem} also has an eigenvalue sequence
$$0< \bar\lambda_{1,h}\leq \bar\lambda_{2,h}\leq\cdots\leq \bar\lambda_{N_h,h},$$
and the corresponding discrete eigenfunction sequence
$$\bar u_{1,h},\bar u_{2,h},\cdots,\bar u_{N_h,h},$$
with the property $a_h(\bar u_{i,h},\bar u_{j,h})=\delta_{ij}, 1\leq i,j\leq N_h$, where $N_h$ is the dimension of $V_h$.

In order to state the error estimate for the eigenpair approximation by the CR finite element method, we define the CR finite element projection $\mathcal P_h: V\mapsto V_h$ as follows
\begin{eqnarray}\label{Energy_Projection first}
a_h(\mathcal P_h u, v_h) = \lambda b(u,v_h),\ \ \ \ \forall v_h\in V_h.
\end{eqnarray}
It is obvious that the finite element projection operator $\mathcal P_h$ has the following error estimates.

\begin{lemma}[\cite{MR2373954,lin2012asymptotic}]\label{Lemma_Error_Estimate_CR}
Assume the source equation corresponding to the eigenvalue problem has $H^{2}(\Omega)$ regularity. Then the following error estimates hold
\begin{eqnarray}
\left\|u-\mathcal P_h u\right\|_{a,h} &\leq &C_1h\|u\|_{2},\label{Delta_V_h_P_h}\\
\left\|u-\mathcal P_h u\right\|_{b} &\leq& C_2h^{2}\|u\|_{2}.\label{Aubin_Nitsche_Estimate}
\end{eqnarray}
\end{lemma}
%-------------------------------------------------------------------------------------------------
Before stating error estimates of the CR finite element method for the eigenvalue problem,
we introduce the following lemma.
\begin{lemma}\label{Strang_Lemma}
For any eigenpair $(\lambda,u)$ of (\ref{weak_eigenvalue_problem}), the following equality holds
\begin{eqnarray*}\label{Strang_Equality}
(\bar\lambda_{j,h}-\lambda)b(\mathcal P_hu,\bar u_{j,h})
=\lambda b(u-\mathcal P_hu,\bar u_{j,h}),\ \ \ j = 1, \cdots, N_h.
\end{eqnarray*}
\end{lemma}
%-------------------------------------------------------------------------------------------------
\begin{proof}
Since $-\lambda b(\mathcal P_h u,\bar u_{j,h})$ appears on both sides, we only need to prove that
\begin{eqnarray*}
\bar\lambda_{j,h}b(\mathcal P_h u,\bar u_{j,h})=\lambda b(u,\bar u_{j,h}).
\end{eqnarray*}
From (\ref{weak_eigenvalue_problem}), (\ref{diseigenproblem}) and (\ref{Energy_Projection first}),
the following equalities hold
\begin{eqnarray*}
\bar\lambda_{j,h}b(\mathcal P_h u,\bar u_{j,h}) = a_h(\mathcal P_h u,\bar u_{j,h})= \lambda b(u,\bar u_{j,h}).
\end{eqnarray*}
Then the proof is completed.
\end{proof}

Now, let us consider the error estimates for the first
$k$ eigenpair approximations associated with $\bar\lambda_{1,h}\leq \cdots \leq \bar\lambda_{k,h}$.
For the following analysis in this paper, we define $\mu_i=1/\lambda_i$ for $i=1,2,\cdots$, and
$\bar\mu_{i, h} = 1/\bar\lambda_{i,h}$ for $i=1, \cdots, N_h$.

\begin{theorem}\label{Error_Estimate_Theorem_k}
Let us define the spectral projection $\bar F_{k,h}: V_h + V \mapsto  {\rm span}\{\bar u_{1,h}, \cdots, \bar u_{k,h}\}$
as follows
\begin{eqnarray}
a_h(\bar F_{k,h}w, \bar u_{i,h}) = a_h(w, \bar u_{i,h}), \ \ \ i=1, \cdots, k\ \ {\rm for}\ w\in V_h + V.
\end{eqnarray}
Then the associated exact eigenfunctions $u_1, \cdots, u_k$ of eigenvalue problem (\ref{weak_eigenvalue_problem}) have the following error estimates
\begin{eqnarray}\label{Theo Energy_Error_Estimate_k}
\left\|u_i - \bar F_{k,h} u_i\right\|_{a,h} \leq 2\|u_i-\mathcal P_h u_i\|_{a,h}
+ \frac{\sqrt{\bar\mu_{k+1,h}}}{\delta_{k,i,h}}\left\|u_i-\mathcal P_hu_i\right\|_{b}, \ \  1\leq i\leq k,
\end{eqnarray}
where $\delta_{k,i,h}$ is defined as follows
\begin{eqnarray}\label{Definition_Delta_k_i_0}
\delta_{k,i,h} = \min_{k<j\leq N_h}\left|\frac{1}{\bar\lambda_{j,h}}-\frac{1}{\lambda_i}\right|.
\end{eqnarray}
Furthermore, these $k$ exact eigenfunctions have the following error estimate in $\left\|\cdot\right\|_{b}$-norm
\begin{eqnarray}\label{Theo L2_Error_Estimate_k_0}
\left\|u_i-\bar F_{k,h} u_i\right\|_{b} \leq \left(2+\frac{\bar\mu_{k+1,h}}{\delta_{k,i,h}}\right)
\left\|u_i-\mathcal P_hu_i\right\|_{b}, \ \ \ 1\leq i\leq k.
\end{eqnarray}
\end{theorem}
\begin{proof}
Since $(I-\bar F_{k,h})\mathcal P_hu_i\in V_h$ and
$(I-\bar F_{k,h})\mathcal P_hu_i\in {\rm span}\{\bar u_{k+1,h},\cdots, \bar u_{N_h,h}\}$,
the following orthogonal expansion holds
\begin{eqnarray}\label{Orthogonal_Decomposition_k}
(I-\bar F_{k,h})\mathcal P_hu_i=\sum_{j=k+1}^{N_h}\alpha_j\bar u_{j,h},
\end{eqnarray}
where $\alpha_j=a_h(\mathcal P_hu_i,\bar u_{j,h})$. From Lemma \ref{Strang_Lemma}, we have
\begin{eqnarray}\label{Alpha_Estimate}
\alpha_j&=&a_h(\mathcal P_hu_i,\bar u_{j,h}) = \bar\lambda_{j,h} b\big(\mathcal P_hu_i,\bar u_{j,h}\big)
=\frac{\bar\lambda_{j,h}\lambda_i}{\bar\lambda_{j,h}-\lambda_i}b\big(u_i-\mathcal P_hu_i,\bar u_{j,h}\big)\nonumber\\
&=&\frac{1}{\mu_i-\bar\mu_{j,h}} b\big(u_i-\mathcal P_hu_i,\bar u_{j,h}\big).
\end{eqnarray}
From the orthogonal property of eigenfunctions $\bar u_{1,h},\cdots, \bar u_{N_h,h}$, we acquire
\begin{eqnarray*}
1 = a_h(\bar u_{j,h},\bar u_{j,h}) = \bar\lambda_{j,h} b(\bar u_{j,h},\bar u_{j,h})
= \bar\lambda_{j,h}\left\|\bar u_{j,h}\right\|_{b}^2,
\end{eqnarray*}
which leads to the following property
\begin{eqnarray}\label{Equality_u_j}
\left\|\bar u_{j,h}\right\|_{b}^2=\frac{1}{\bar\lambda_{j,h}}=\bar\mu_{j,h}.
\end{eqnarray}
Because of (\ref{diseigenproblem}) and the definitions of eigenfunctions $\bar u_{1,h},\cdots, \bar u_{N_h,h}$,
we obtain the following equalities
\begin{eqnarray}\label{Orthonormal_Basis}
a_h(\bar u_{j,h},\bar u_{k,h})=\delta_{jk},
\ \ \ \ \ b\left(\frac{\bar u_{j,h}}{\left\|\bar u_{j,h}\right\|_{b}},
\frac{\bar u_{k,h}}{\left\|\bar u_{k,h}\right\|_{b}}\right)=\delta_{jk},\ \ \ 1\leq j,k\leq N_h.
\end{eqnarray}
Then due to (\ref{Orthogonal_Decomposition_k}), (\ref{Alpha_Estimate}), (\ref{Equality_u_j}) and (\ref{Orthonormal_Basis}),
we have the following estimates
\begin{eqnarray}\label{Equality_4_i}
&&\left\|(I-\bar F_{k,h})\mathcal P_hu_i\right\|_{a,h}^2 = \left\|\sum_{j=k+1}^{N_h}\alpha_j\bar u_{j,h}\right\|_{a,h}^2
= \sum_{j=k+1}^{N_h}\alpha_j^2\nonumber\\
&&=\sum_{j=k+1}^{N_h} \left(\frac{1}{\mu_i-\bar\mu_{j,h}}\right)^2 b\big(u_i-\mathcal P_hu_i,\bar u_{j,h}\big)^2\nonumber\\
&&\leq\frac{1}{\delta_{k,i,h}^2}\sum_{j=k+1}^{N_h}\left\|\bar u_{j,h}\right\|_{b}^2
b\left(u_i-\mathcal P_hu_i,\frac{\bar u_{j,h}}{\left\|\bar u_{j,h}\right\|_{b}}\right)^2\nonumber\\
&&=\frac{1}{\delta_{k,i,h}^2}\sum_{j=k+1}^{N_h}\bar\mu_{j,h}
b\left(u_i-\mathcal P_hu_i,\frac{\bar u_{j,h}}{\left\|\bar u_{j,h}\right\|_{b}}\right)^2\nonumber\\
&&
\leq \frac{\bar\mu_{k+1,h}}{\delta_{k,i,h}^2}\sum_{j=k+1}^{N_h}
b\left(u_i-\mathcal P_hu_i,\frac{\bar u_{j,h}}{\left\|\bar u_{j,h}\right\|_{b}}\right)^2 \nonumber\\
&&\leq \frac{\bar\mu_{k+1,h}}{\delta_{k,i,h}^2}\left\|u_i-\mathcal P_hu_i\right\|_{b}^2,
\end{eqnarray}
where the last inequality holds since $\frac{\bar u_{1,h}}{\left\|\bar u_{1,h}\right\|_{b}}$, $\cdots$,
$\frac{\bar u_{j,h}}{\left\|\bar u_{j,h}\right\|_{b}}$ are the normal orthogonal basis for the space $V_h$
in the sense of the inner product $b(\cdot, \cdot)$.

From (\ref{Equality_4_i}),  the following inequality holds
\begin{eqnarray}\label{Equality_5_k}
\left\|(I-\bar F_{k,h})\mathcal P_hu_i\right\|_{a,h}
\leq\frac{\sqrt{\bar\mu_{k+1,h}}}{\delta_{k,i,h}}\left\|u_i-\mathcal P_hu_i\right\|_{b}.
\end{eqnarray}
Since $\bar F_{k,h}$ is the spectral projection operator with respect to $a_h(\cdot, \cdot)$, we write
\begin{equation}\label{Fk}
\|\bar F_{k,h}\|_{a,h}\leq 1
\end{equation}
From (\ref{Equality_5_k}), \eqref{Fk} and the triangle inequality,  it follows that
\begin{eqnarray*}
&&\left\|u_i-\bar F_{k,h}u_i\right\|_{a,h}\leq \left\|u_i-\mathcal P_hu_i\right\|_{a,h}+  \left\|(I-\bar F_{k,h})\mathcal P_hu_i\right\|_{a,h} + \left\|\bar F_{k,h}(\mathcal P_h-I)u_i\right\|_{a,h}\nonumber\\
&&\leq  \left\|u_i-\mathcal P_hu_i\right\|_{a,h} + \left\|(I-\bar F_{k,h})\mathcal P_hu_i\right\|_{a,h} +
\left\|\bar F_{k,h}\right\|_{a,h}\left\|(\mathcal P_h-I)u_i\right\|_{a,h}\nonumber\\
&&\leq 2\|u_i-\mathcal P_h u_i\|_{a,h} + \frac{\sqrt{\bar\mu_{k+1,h}}}{\delta_{k,i,h}}\left\|u_i-\mathcal P_hu_i\right\|_{b}.
\end{eqnarray*}
This is the desired result (\ref{Theo Energy_Error_Estimate_k}).
%-----------------------------------------------------------------------------------------------------

Similarly, with the help of  (\ref{Orthogonal_Decomposition_k}),
(\ref{Alpha_Estimate}), (\ref{Equality_u_j}) and (\ref{Orthonormal_Basis}),
we have the following estimates
\begin{eqnarray*}
&&\left\|(I-\bar F_{k,h})\mathcal P_hu_i\right\|_{b}^2 = \left\|\sum_{j=k+1}^{N_h}\alpha_j\bar u_{j,h}\right\|_{b}^2
= \sum_{j=k+1}^{N_h}\alpha_j^2\left\|\bar u_{j,h}\right\|_{b}^2\nonumber\\
&&=\sum_{j=k+1}^{N_h} \left(\frac{1}{\mu_i-\bar\mu_{j,h}}\right)^2
b\big(u_i-\mathcal P_hu_i,\bar u_{j,h}\big)^2\left\|\bar u_{j,h}\right\|_{b}^2\nonumber\\
&&\leq\frac{1}{\delta_{k,i,h}^2}\sum_{j=k+1}^{N_h}\left\|\bar u_{j,h}\right\|_{b}^4\
b\left(u_i-\mathcal P_hu_i,\frac{\bar u_{j,h}}{\left\|\bar u_{j,h}\right\|_{b}}\right)^2\nonumber\\
&&=\frac{1}{\delta_{k,i,h}^2}\sum_{j=k+1}^{N_h}\bar\mu_{j,h}^2
b\left(u_i-\mathcal P_hu_i,\frac{\bar u_{j,h}}{\left\|\bar u_{j,h}\right\|_{b}}\right)^2
\leq \frac{\bar\mu_{k+1,h}^2}{\delta_{k,i,h}^2}\left\|u_i-\mathcal P_hu_i\right\|_{b}^2,
\end{eqnarray*}
which leads to the inequality
\begin{eqnarray}\label{Equality_8_k}
\left\|(I-\bar F_{k,h})\mathcal P_hu_i\right\|_{b} \leq \frac{\bar\mu_{k+1,h}}{\delta_{k,i,h}}
\left\|u_i-\mathcal P_hu_i\right\|_{b}.
\end{eqnarray}
According to the definitions \eqref{ahb} of the norms $\|\cdot\|_{a,h}$ and $\|\cdot\|_b$ and the reference \cite{Conway}, we know that the norm $\|\cdot\|_{a,h}$ is relatively compact with respect to the norm $\|\cdot\|_b$. Combined with \eqref{Fk}, we get $\|\bar F_{k,h}\|_{b}\leq 1$. And from (\ref{Equality_8_k}) and the triangle inequality, we find the
following error estimates for the eigenfunction approximations in the $\left\|\cdot\right\|_{b}$-norm
\begin{eqnarray*}\label{Inequality_11}
&&\left\|u_i-\bar F_{k,h}u_i\right\|_{b}\leq \left\|u_i-\mathcal P_hu_i\right\|_{b}
+ \left\|(I-\bar F_{k,h}) \mathcal P_hu_i\right\|_{b} + \left\|\bar F_{k,h}(\mathcal P_hu_i-u_i)\right\|_{b}\nonumber\\
&&\leq\left(1+\|\bar F_{k,h}\|_{b}\right) \left\|\mathcal P_hu_i-u_i \right\|_{b} +
\left\|(I-\bar F_{k,h})\mathcal P_hu_i\right\|_{b}\nonumber\\
&&\leq\left(2+\frac{\bar\mu_{k+1,h}}{\delta_{k,i,h}}\right)
\left\|u_i-\mathcal P_hu_i\right\|_{b}.
\end{eqnarray*}
This is the second desired result (\ref{Theo L2_Error_Estimate_k_0}) and the proof is completed.
\end{proof}

In order to make sense of the estimates (\ref{Theo Energy_Error_Estimate_k}) and
(\ref{Theo L2_Error_Estimate_k_0}), and for simplicity of notation, we assume that the eigenvalue gap $\delta_{k,i,h}$
has a uniform lower bound which is denoted by $\delta_{k,i}$ (which can be seen as the
``true" separation of the eigenvalues $\lambda_1, \cdots, \lambda_k$ from the unwanted eigenvalues)
in the following parts of this paper. This assumption is reasonable when the mesh size is small enough.
Then we have the following convergence order based on Theorem \ref{Error_Estimate_Theorem_k} and
the convergence results of CR finite element method for boundary value problems.
\begin{corollary}\label{Error_Estimate_Corollary_k}
Under the conditions of Lemma \ref{Lemma_Error_Estimate_CR}, Theorem \ref{Error_Estimate_Theorem_k}
and $\delta_{k,i,h}$ having a uniform lower bound $\delta_{k,i}$,  the following error estimates hold
\begin{eqnarray}
\left\|u_i - F_{k,h} u_i\right\|_{a,h} &\leq& C_3h\|u\|_{2}, \ \ \ \  1\leq i\leq k,\label{Energy_Error_Estimate_k}\\
\left\|u_i-F_{k,h}u_i\right\|_{b} &\leq& C_4h^{2}\|u\|_{2},\ \ \ 1\leq i\leq k.\label{L2_Error_Estimate_k}
\end{eqnarray}
\end{corollary}
%-------------------------------------------------------------------------------------------------
The following theorem gives the error estimates for the one eigenpair approximation and the proof is similar
to that of Theorem \ref{Error_Estimate_Theorem_k}.
%-------------------------------------------------------------------------------------------------
\begin{theorem}\label{Error_Estimate_Theorem_Old}
Let  $(\lambda,u)$ denote an exact eigenpair of the eigenvalue problem (\ref{weak_eigenvalue_problem}).
Assume the eigenpair approximation $(\bar\lambda_{i,h},\bar u_{i,h})$ has the property that
$\bar\mu_{i,h}=1/\bar\lambda_{i,h}$ is the closest to $\mu=1/\lambda$.
The corresponding spectral projector $E_{i,h}: V_h+V\mapsto {\rm span}\{\bar u_{i,h}\}$
is  defined as follows
\begin{eqnarray*}
a_h(E_{i,h}w,\bar u_{i,h}) = a_h(w,\bar u_{i,h}),\ \ \ \ {\rm for}\  w\in V_h+V.
\end{eqnarray*}
Then the following error estimate holds
\begin{eqnarray}\label{Energy_Error_Estimate_Old}
\left\|u-E_{i,h}u\right\|_{a,h}&\leq& 2\|u-\mathcal P_h u\|_{a,h}
+ \frac{\sqrt{\bar\mu_{1,h}}}{\delta_{\lambda,h}}\left\|u-\mathcal P_hu\right\|_{b},
\end{eqnarray}
where  $\delta_{\lambda,h}$ is defined as follows
\begin{eqnarray}\label{Definition_Delta}
\delta_{\lambda,h} &:=& \min_{j\neq i}|\bar\mu_{j,h}-\mu|=\min_{j\neq i} \left|\frac{1}{\bar\lambda_{j,h}}-\frac{1}{\lambda}\right|.
\end{eqnarray}
Furthermore, the eigenfunction approximation $\bar u_{i,h}$ has the following
error estimate in $\left\|\cdot\right\|_{b}$-norm
\begin{eqnarray}\label{L2_Error_Estimate_Old}
\|u - E_{i,h}u\|_{b} &\leq &  \left(2+\frac{\bar\mu_{1,h}}{\delta_{\lambda,h}}\right)
\left\|u-\mathcal P_hu\right\|_{b}.
\end{eqnarray}
\end{theorem}
\begin{proof}
Since $(I-E_{i,h})\mathcal P_hu\in V_h$ and
$(I-E_{i,h})\mathcal P_hu\in {\rm span}\{\bar u_{1,h}, \cdots, \bar u_{i-1,h}, \bar u_{i+1,h},\cdots, \bar u_{N_h,h}\}$,
the following orthogonal expansion holds
\begin{eqnarray}\label{Orthogonal_Decomposition_1}
(I-E_{i,h})\mathcal P_hu=\sum_{j\neq i}\alpha_j\bar u_{j,h},
\end{eqnarray}
where $\alpha_j=a_h(\mathcal P_hu,\bar u_{j,h})$ has the same equality (\ref{Alpha_Estimate}).

Then due to (\ref{Alpha_Estimate}), (\ref{Equality_u_j}),  (\ref{Orthonormal_Basis}) and  (\ref{Orthogonal_Decomposition_1}),
we have the following estimates
\begin{eqnarray}\label{Equality_4_i_1}
&&\left\|(I-E_{i,h})\mathcal P_hu\right\|_{a,h}^2 = \left\|\sum_{j\neq i}\alpha_j\bar u_{j,h}\right\|_{a,h}^2
= \sum_{j\neq i}\alpha_j^2\nonumber\\
&&=\sum_{j\neq i} \left(\frac{1}{\mu-\bar\mu_{j,h}}\right)^2 b\big(u-\mathcal P_hu,\bar u_{j,h}\big)^2
\leq\frac{1}{\delta_{\lambda,h}^2}\sum_{j\neq i}\left\|\bar u_{j,h}\right\|_{b}^2
b\left(u-\mathcal P_hu,\frac{\bar u_{j,h}}{\left\|\bar u_{j,h}\right\|_{b}}\right)^2\nonumber\\
&&=\frac{1}{\delta_{\lambda,h}^2}\sum_{j\neq i}\bar\mu_{j,h}
b\left(u-\mathcal P_hu,\frac{\bar u_{j,h}}{\left\|\bar u_{j,h}\right\|_{b}}\right)^2\nonumber\\
&&
\leq \frac{\bar\mu_{1,h}}{\delta_{\lambda,h}^2}\sum_{j\neq i}b\left(u-\mathcal P_hu,\frac{\bar u_{j,h}}{\left\|\bar u_{j,h}\right\|_{b}}\right)^2\\
&&\leq \frac{\bar\mu_{1,h}}{\delta_{\lambda,h}^2}\left\|u-\mathcal P_hu\right\|_{b}^2,
\end{eqnarray}
where the last inequality holds since $\frac{\bar u_{1,h}}{\left\|\bar u_{1,h}\right\|_{b}}$, $\cdots$,
$\frac{\bar u_{j,h}}{\left\|\bar u_{j,h}\right\|_{b}}$ are the normal orthogonal basis for the space $V_h$
in the sense of the inner product $b(\cdot, \cdot)$.

From (\ref{Equality_4_i_1}),  the following inequality holds
\begin{eqnarray}\label{Equality_5_1}
\left\|(I-E_{i,h})\mathcal P_hu\right\|_{a,h}
\leq\frac{\sqrt{\bar\mu_{1,h}}}{\delta_{\lambda,h}}\left\|u-\mathcal P_hu\right\|_{b}.
\end{eqnarray}
Since $E_{i,h}$ is the spectral projection operator with respect to $a_h(\cdot, \cdot)$, we get $\|E_{i,h}\|_{a,h}\leq 1$. And from (\ref{Equality_5_1}) and the triangle inequality,  it follows that
\begin{eqnarray*}
&&\left\|u-E_{i,h}u\right\|_{a,h}\leq \left\|u-\mathcal P_hu\right\|_{a,h}+  \left\|(I-E_{i,h})\mathcal P_hu\right\|_{a,h} + \left\|E_{i,h}(\mathcal P_h-I)u\right\|_{a,h}\nonumber\\
&&\leq  \left\|u-\mathcal P_hu\right\|_{a,h} + \left\|(I-E_{i,h})\mathcal P_hu\right\|_{a,h} +
\left\|E_{i,h}\right\|_{a,h}\left\|(\mathcal P_h-I)u\right\|_{a,h}\nonumber\\
&&\leq 2\|u-\mathcal P_h u\|_{a,h} + \frac{\sqrt{\bar\mu_{1,h}}}{\delta_{\lambda,h}}\left\|u-\mathcal P_hu\right\|_{b}.
\end{eqnarray*}
This is the desired result (\ref{Energy_Error_Estimate_Old}).
%-----------------------------------------------------------------------------------------------------

Similarly, with the help of  (\ref{Alpha_Estimate}), (\ref{Equality_u_j}), (\ref{Orthonormal_Basis}) and (\ref{Orthogonal_Decomposition_1}), we have the following estimates
\begin{eqnarray*}
&&\left\|(I-E_{i,h})\mathcal P_hu\right\|_{b}^2 = \left\|\sum_{j\neq i}\alpha_j\bar u_{j,h}\right\|_{b}^2
= \sum_{j\neq i}\alpha_j^2\left\|\bar u_{j,h}\right\|_{b}^2\nonumber\\
&&=\sum_{j\neq i} \left(\frac{1}{\mu-\bar\mu_{j,h}}\right)^2
b\big(u-\mathcal P_hu,\bar u_{j,h}\big)^2\left\|\bar u_{j,h}\right\|_{b}^2\nonumber\\
&&\leq\frac{1}{\delta_{\lambda, h}^2}\sum_{j\neq i}\left\|\bar u_{j,h}\right\|_{b}^4\
b\left(u-\mathcal P_hu,\frac{\bar u_{j,h}}{\left\|\bar u_{j,h}\right\|_{b}}\right)^2\nonumber\\
&&\leq \frac{\bar\mu_{1,h}^2}{\delta_{\lambda,h}^2}\left\|u-\mathcal P_hu\right\|_{b}^2,
\end{eqnarray*}
which leads to the inequality
\begin{eqnarray}\label{Equality_8_1}
\left\|(I-E_{i,h})\mathcal P_hu\right\|_{b} \leq \frac{\bar\mu_{k+1,h}}{\delta_{k,i,h}}
\left\|u-\mathcal P_hu\right\|_{b}.
\end{eqnarray}
Due to the definitions \eqref{ahb} of the norms $\|\cdot\|_{a, h}$ and $\|\cdot\|_b$, we illustrate that the norm $\|\cdot\|_{a,h}$ is relatively compact with respect to the norm $\|\cdot\|_b$. And together with $\|E_{i,h}\|_{a,h}\leq 1$, we obtain $\|E_{i,h}\|_{b}\leq 1$. From (\ref{Equality_8_1}), $\|E_{i,h}\|_{b}\leq 1$ and the triangle inequality, we find the
following error estimates for the eigenfunction approximations in the $\left\|\cdot\right\|_{b}$-norm
\begin{eqnarray*}\label{Inequality_11}
&&\left\|u-E_{i,h}u\right\|_{b}\leq \left\|u-\mathcal P_hu\right\|_{b} + \left\|(I-E_{i,h})
\mathcal P_hu\right\|_{b} + \left\|E_{i,h}(\mathcal P_hu-u)\right\|_{b}\nonumber\\
&&\leq\left(1+\|E_{i,h}\|_{b}\right) \left\|\mathcal P_hu-u \right\|_{b} +
\left\|(I-E_{i,h})\mathcal P_hu\right\|_{b}\leq\left(2+\frac{\bar\mu_{1,h}}{\delta_{\lambda,h}}\right)
\left\|u-\mathcal P_hu\right\|_{b}.
\end{eqnarray*}
This is the second desired result (\ref{L2_Error_Estimate_Old}) and the proof is completed.
\end{proof}
%------------------------------------------------------------------------------------------------------------------
Similarly, in order to make sense of the estimates (\ref{Energy_Error_Estimate_Old}) and
(\ref{L2_Error_Estimate_Old}), and for simplicity of notation,
we assume that the eigenvalue gap $\delta_{\lambda,h}$ defined by (\ref{Definition_Delta})
has also a uniform lower bound which is denoted by $\delta_\lambda$ (which can be seen as the
``true" separation of the eigenvalue $\lambda$ from others) in the following parts of this paper.
This assumption is also reasonable when the mesh size is small enough.
Then we also have the following convergence result
based on Theorem \ref{Error_Estimate_Theorem_Old} and the convergence results of CR finite element method for
boundary value problems.
\begin{corollary}\label{Error_Estimate_Corollary}
Under the conditions of Lemma \ref{Lemma_Error_Estimate_CR}, Theorem \ref{Error_Estimate_Theorem_Old} and
$\delta_{\lambda,h}$ having a uniform lower bound $\delta_\lambda$, the following error estimates hold
\begin{eqnarray}
\left\|u-E_{i,h}u\right\|_{a,h}&\leq& C_5h\|u\|_{2},\label{Energy_Error_Estimate}\\
\left\|u-E_{i,h}u\right\|_{b}&\leq& C_6h^{2}\|u\|_{2}. \label{L2_Error_Estimate}
\end{eqnarray}
\end{corollary}

\section{Augmented subspace method and its error estimates}\label{Section_3}
In this section, we first present the augmented subspace method for solving the eigenvalue problem \eqref{weak_eigenvalue_problem} based on CR element.
This method contains solving the auxiliary linear boundary value problem
in the fine finite element space $V_h$ and the eigenvalue problem on the
augmented subspace $V_{H,h}$ which is built by the coarse finite element space $V_H$
and a finite element function in the fine finite element space $V_h$. In order to eliminate the compatibility error of the CR element space $V_H$, we use the conforming linear finite element space $W_H$ to construct the augmented subspace $V_{H,h}$.
Then, the new convergence analysis is given for this augmented subspace method.

In order to design the augmented subspace method, we first generate a coarse mesh $\mathcal{T}_H$
with the mesh size $H$ and the coarse linear finite element space $W_H$ is
defined on the mesh $\mathcal{T}_H$. For the positive integer $\ell$ and some given eigenfunction approximations $u_{1,h}^{(\ell)},\cdots, u_{k,h}^{(\ell)}$
which are the approximations for the first $k$ eigenfunctions $\bar u_{1,h},\cdots, \bar u_{k,h}$
of \eqref{weak_eigenvalue_problem}, we can do the following augmented subspace iteration step
which is defined by Algorithm \ref{alg:augk_NC} to improve the accuracy of $u_{1,h}^{(\ell)},\cdots, u_{k,h}^{(\ell)}$, where the superscript with parentheses denotes the number of iteration steps of the augmented subspace method.
\begin{algorithm}[hbt!]
\caption{Augmented subspace method for the first $k$ eigenpairs}\label{alg:augk_NC}
\begin{enumerate}
\item For $\ell=1$, we define $\widehat u_{i,h}^{(\ell)}=u_{i,h}^{(\ell)}$, $i=1,\cdots, k$,   and
the augmented subspace $V_{H,h}^{(\ell)} = W_H +{\rm span}\{\widehat u_{1,h}^{(\ell)}, \cdots, \widehat u_{k,h}^{(\ell)}\}$.
Then solve the following eigenvalue problem:
Find $(\lambda_{i,h}^{(\ell)},u_{i,h}^{(\ell)})\in \mathbb{R}\times V_{H,h}^{(\ell)}$
such that $a_h(u_{i,h}^{(\ell)},u_{i,h}^{(\ell)})=1$ and
\begin{equation}\label{EigenProblem_k_NC_1}
a_h(u_{i,h}^{(\ell)},v_{H,h}) = \lambda_{i,h}^{(\ell)}b(u_{i,h}^{(\ell)},v_{H,h}),
\ \ \ \ \ \forall v_{H,h}\in V_{H,h}^{(\ell)},\ \ \ i=1, \cdots, k.
\end{equation}

\item Solve the following linear boundary value problems:
Find $\widehat{u}_{i,h}^{(\ell+1)}\in V_h$ such that
\begin{equation}\label{eq:Linear_Equation_k_NC}
a_h(\widehat{u}_{i,h}^{(\ell+1)},v_h) = \lambda_{i,h}^{(\ell)}b(u_{i,h}^{(\ell)},v_h),
\ \  \forall v_h\in V_h,\ \ \ i=1, \cdots, k.
\end{equation}

\item Define the augmented subspace $V_{H,h}^{(\ell+1)} = W_H +
{\rm span}\{\widehat{u}_{1,h}^{(\ell+1)}, \cdots, \widehat u_{k,h}^{(\ell+1)}\}$ and solve the following eigenvalue problem:
Find $(\lambda_{i,h}^{(\ell+1)},u_{i,h}^{(\ell+1)})\in \mathbb{R}\times V_{H,h}^{(\ell+1)}$
such that $a_h(u_{i,h}^{(\ell+1)},u_{i,h}^{(\ell+1)})=1$ and
\begin{equation}\label{Aug_Eigenvalue_Problem_k}
a_h(u_{i,h}^{(\ell+1)},v_{H,h}) = \lambda_{i,h}^{(\ell+1)}b(u_{i,h}^{(\ell+1)},v_{H,h}),
\ \ \ \ \ \forall v_{H,h}\in V_{H,h}^{(\ell+1)},\ \ \ i=1, \cdots, k.
\end{equation}
Solve (\ref{Aug_Eigenvalue_Problem_k})  to obtain $(\lambda_{1,h}^{(\ell+1)},u_{1,h}^{(\ell+1)}), \cdots,
(\lambda_{k,h}^{(\ell+1)},u_{k,h}^{(\ell+1)})$.
\item Set $\ell=\ell+1$ and go to Step 2 for the next iteration until convergence.
\end{enumerate}
\end{algorithm}

It should be noted that, although $V_h$ is nonconforming, we can still get the following nested relationship because of conforming linear finite element space $W_H$,
\begin{equation}\label{relation_NC}
W_H\subset V_{H,h}\subset V_h.
\end{equation}
In order to derive the algebraic error estimates of Algorithm \ref{alg:augk_NC}, we first concern with the error estimates of the projection operator $\mathcal P_{H,h}: V_{h}\rightarrow V_{H,h}$. The definition of $\mathcal P_{H,h}$ is given as follows.
\begin{equation}\label{Energy_Projection}
a_h\left(\mathcal{P}_{H,h} v_h, v_{H,h}\right) = a_h\left(v_h,v_{H,h}\right),\quad \forall v_{H,h}\in V_{H,h},\ \ \text{for}\ \ v_h\in V_h.
\end{equation}
From \eqref{relation_NC}, it follows that
\begin{equation}\label{eq:ZuiYou_NC}
\|v_h - \mathcal P_{H,h}v_h\|_{a,h} = \inf_{v_{H,h}\in V_{H,h}}\|v_h - v_{H,h}\|_{a,h},
\end{equation}
and
\begin{align*}
\|\bar u_h - \mathcal P_{H,h}\bar u_h\|_b &= \sup_{f\in L^2(\Omega),\|f\|_b=1} b\left(\bar u_h - \mathcal P_{H,h} \bar u_h,f\right)\\
&=  \sup_{f\in L^2(\Omega),\|f\|_b=1} a_h\left(\bar u_h - \mathcal P_{H,h} \bar u_h,T_h f\right)\\
&= \sup_{f\in L^2(\Omega),\|f\|_b=1} a_h\left(\bar u_h - \mathcal P_{H,h} \bar u_h,T_h f - v_{H,h}\right),\quad \forall v_{H,h}\in V_{H,h}.
\end{align*}
Thus,
\begin{equation}\label{eq:NiCai_NC}
\|\bar u_h - \mathcal P_{H,h}\bar u_h\|_b \le \eta_a\left(V_{H,h}\right) \|\bar u_h - \mathcal P_{H,h}\bar u_h\|_{a,h},
\end{equation}
where
\begin{equation}\label{eq:etaa_NC}
\eta_a\left(V_{H,h}\right) = \sup_{f\in L^2(\Omega),\|f\|_b=1}\inf_{v_{H,h}\in V_{H,h}}\|T_h f-v_{H,h}\|_{a,h}.
\end{equation}
By virtue of \eqref{relation_NC}, we provide
\begin{equation}\label{eq:etarelation_NC}
\eta_a\left(V_{H,h}\right) \le \eta_a\left(W_{H}\right).
\end{equation}

\begin{theorem}\label{thm:algerr_k_NC}
Let us define the spectral projection $F_{k,h}^{(m)}: V_h\mapsto {\rm span}\{u_{1,h}^{(m)}, \cdots, u_{k,h}^{(m)}\}$ for any integer $m \geq 1$
as follows
\begin{eqnarray}
a_h(F_{k,h}^{(m)}w_h, u_{i,h}^{(m)}) = a_h(w_h, u_{i,h}^{(m)}), \ \ \ i=1, \cdots, k\ \ {\rm for}\ w_h\in V_h.
\end{eqnarray}
Then the exact eigenfunctions $\bar u_{1,h},\cdots, \bar u_{k,h}$ of (\ref{diseigenproblem}) and the eigenfunction
approximations $u_{1,h}^{(\ell+1)}$, $\cdots$,  $u_{k,h}^{(\ell+1)}$ from Algorithm \ref{alg:augk_NC} with the integer $\ell \geq 1$ have the following error estimate
\begin{equation}\label{Error_Estimate_Inverse}
\left\|\bar u_{i,h} -F_{k,h}^{(\ell+1)}\bar u_{i,h} \right\|_{a,h} \leq
\bar\lambda_{i,h} \sqrt{1+\frac{\eta_a^2(W_H)}{\lambda_{k+1}\big(\delta_{k,i}^{(\ell+1)}\big)^2}}
\left(1+\frac{1}{\lambda_{k+1}\delta_{k,i}^{(\ell)}}\right)\eta_a^2(W_H)\left\|\bar u_{i,h} - F_{k,h}^{(\ell)}\bar u_{i,h} \right\|_{a,h},
\end{equation}
where $\eta_a\left(W_H\right)$ is defined by \eqref{eq:etaa_NC}. Furthermore, the following $\left\|\cdot\right\|_b$-norm error estimate holds
\begin{eqnarray}\label{L2_Error_Estimate_Algorithm_1}
&&\left\|\bar u_{i,h} -F_{k,h}^{(\ell+1)}\bar u_{i,h} \right\|_b\leq \left(1+\frac{1}{\lambda_{k+1}\delta_{k,i}^{(\ell+1)}}\right)\eta_a(W_H) \left\|\bar u_{i,h} -F_{k,h}^{(\ell+1)}\bar u_{i,h}\right\|_{a,h}.
\end{eqnarray}
Here denote by $\delta_{k,i}^{(\ell +1)}$ the uniform lower bound of the eigenvalue gap $\delta_{k,i,h}^{(\ell +1)}$, which is defined as
\begin{equation}
\delta_{k,i,h}^{(\ell +1)} := \min_{k < j \le N_h} \left|\dfrac{1}{\lambda_{j,h}^{(\ell +1)}} - \dfrac{1}{\lambda_i}\right|.
\end{equation}
\end{theorem}

\begin{proof}
First, let us consider the error estimate $\left\|\bar u_{i,h}- F_{k,h}^{(\ell)}\bar u_{i,h}\right\|_b$. Due to Algorithm \ref{alg:augk_NC},
we know that the approximations $u_{1,h}^{(\ell)}, \cdots, u_{k,h}^{(\ell)}$ come from (\ref{EigenProblem_k_NC_1}) (the case that $\ell=1$) or (\ref{Aug_Eigenvalue_Problem_k})
(the case that $\ell >1$). Similarly with the derivation in the case of the conforming finite element method (refer to Theorem 3.1 in \cite{MR4530254}), for both cases, there exist exact eigenfunctions $\bar u_{1,h},\cdots, \bar u_{k,h}$
such that the following error estimates for the eigenvector approximations
$u_{1,h}^{(\ell)},\cdots,  u_{k,h}^{(\ell)}$ hold for $i=1, \cdots, k$

\begin{eqnarray}\label{Inequality_13}
\left\|\bar u_{i,h}- F_{k,h}^{(\ell)}\bar u_{i,h}\right\|_b &\leq \left(1+\frac{\mu_{k+1}}{\delta_{k,i}^{(\ell)}}\right)\eta_a(V_{H,h}^{(\ell)})\left\|\bar u_{i,h}- F_{k,h}^{(\ell)}\bar u_{i,h}\right\|_{a,h} \nonumber\\
&\leq \left(1+\frac{\mu_{k+1}}{\delta_{k,i}^{(\ell)}}\right)\eta_a(W_H)\left\|\bar u_{i,h}- F_{k,h}^{(\ell)}\bar u_{i,h}\right\|_{a,h},
\end{eqnarray}
where we have used the inequality $\eta_a(V_{H,h}^{(\ell)})\leq \eta_a(W_H)$ since $W_H\subset V_{H,h}^{(\ell)}$.

From the definition of the spectral projection $F_{k,h}^{(\ell)}$.
Then there exist $k$ real numbers $q_1, \cdots, q_k \in \mathbb R$ such that $F_{k,h}^{(\ell)}\bar u_{i,h}$ has the following expansion
\begin{eqnarray}\label{Expansion_L2}
F_{k,h}^{(\ell)}\bar u_{i,h} = \sum_{j=1}^k q_ju_{j,h}^{(\ell)}.
\end{eqnarray}
From \eqref{Energy_Projection}, we obtain the orthogonal property of the projection operator $\mathcal P_{H,h}^{(\ell+1)}$, that is to say,
\begin{equation*}
a_h(\bar u_{i,h}-\mathcal P_{H,h}^{(\ell+1)}\bar u_{i,h}, v_{H,h})=0, \quad \forall v_{H,h}\in V_{H,h}^{(\ell+1)}.
\end{equation*}
Together with the definition of $V_{H,h}^{(\ell+1)}$ in Step 3 of Algorithm \ref{alg:augk_NC}, we supply
\begin{eqnarray*}
&&\left\|\bar u_{i,h} - \mathcal P_{H,h}^{(\ell+1)}\bar u_{i,h}\right\|_{a,h}^2 =
a_h\left(\bar u_{i,h} - \mathcal P_{H,h}^{(\ell+1)}\bar u_{i,h}, \bar u_{i,h} - \mathcal P_{H,h}^{(\ell+1)}\bar u_{i,h}\right)\nonumber\\
&&=a_h\left(\bar u_{i,h}, \bar u_{i,h} - \mathcal P_{H,h}^{(\ell+1)}\bar u_{i,h}\right)\nonumber\\
&&=a_h\left(\bar u_{i,h} - \sum_{j=1}^k\bar\lambda_{i,h}\frac{q_j}{\lambda_{j,h}^{(\ell)}}
\widehat u_{j,h}^{(\ell+1)}, \bar u_{i,h} - \mathcal P_{H,h}^{(\ell+1)}\bar u_{i,h}\right).
\end{eqnarray*}
Because of \eqref{diseigenproblem}, \eqref{eq:Linear_Equation_k_NC} and $V_{H,h}^{(\ell+1)}\subset V_h$, we provide
\begin{eqnarray*}
&&\left\|\bar u_{i,h} - \mathcal P_{H,h}^{(\ell+1)}\bar u_{i,h}\right\|_{a,h}^2 =
\bar\lambda_{i,h} b\left(\bar u_{i,h} - \sum_{j=1}^k\frac{q_j}{\lambda_{j,h}^{(\ell)}}\lambda_{j,h}^{(\ell)}u_{j,h}^{(\ell)}, \bar u_{i,h} - \mathcal P_{H,h}^{(\ell+1)}\bar u_{i,h}\right)\nonumber\\
&&=\bar\lambda_{i,h} b\left(\bar u_{i,h} - \sum_{j=1}^kq_ju_{j,h}^{(\ell)}, \bar u_{i,h} - \mathcal P_{H,h}^{(\ell+1)}\bar u_{i,h}\right),
\end{eqnarray*}
combined with (\ref{Expansion_L2}), then
\begin{eqnarray*}
&&\left\|\bar u_{i,h} - \mathcal P_{H,h}^{(\ell+1)}\bar u_{i,h}\right\|_{a,h}^2 =\bar\lambda_{i,h} b\left(\bar u_{i,h} - F_{k,h}^{(\ell)}\bar u_{i,h}, \bar u_{i,h} - \mathcal P_{H,h}^{(\ell+1)}\bar u_{i,h}\right)\nonumber\\
&&\leq \bar\lambda_{i,h}\left\|\bar u_{i,h} - F_{k,h}^{(\ell)}\bar u_{i,h}\right\|_b\left\|\bar u_{i,h} - \mathcal P_{H,h}^{(\ell+1)}\bar u_{i,h}\right\|_b.
\end{eqnarray*}
And considering \eqref{Inequality_13}, \eqref{eq:NiCai_NC} and the inequality $\eta_a(V_{H,h}^{(\ell+1)})\leq \eta_a(W_H)$, we render
\begin{equation*}
\left\|\bar u_{i,h} - \mathcal P_{H,h}^{(\ell+1)}\bar u_{i,h}\right\|_{a,h}^2 \leq  \bar\lambda_{i,h} \left(1+\frac{1}{\lambda_{k+1}\delta_{k,i}}\right)\eta^2_a(W_H)\left\|\bar u_{i,h}- F_{k,h}^{(\ell)}\bar u_{i,h}\right\|_{a,h} \left\|\bar u_{i,h} - \mathcal P_{H,h}^{(\ell+1)}\bar u_{i,h}\right\|_{a,h},
\end{equation*}
i.e.,
\begin{eqnarray}\label{Inequality_18_k}
\left\|\bar u_{i,h} - \mathcal P_{H,h}^{(\ell+1)}\bar u_{i,h}\right\|_{a,h} \leq \bar\lambda_{i,h} \left(1+\frac{1}{\lambda_{k+1}\delta_{k,i}}\right)\eta^2_a(W_H)\left\|\bar u_{i,h}- F_{k,h}^{(\ell)}\bar u_{i,h}\right\|_{a,h}.
\end{eqnarray}
Since $u_{1,h}^{(\ell+1)}, \cdots, u_{k,h}^{(\ell+1)}$ only come from (\ref{Aug_Eigenvalue_Problem_k})
and the orthogonal property $a_h((I-\mathcal P_{H,h}^{(\ell+1)})\bar u_{i,h}, (I-F_{k,h}^{\ell})P_{H,h}^{(\ell+1)} \bar u_{i,h}) = 0$,
we have for $i=1, \cdots, k$
\begin{eqnarray*}
\left\|\bar u_{i,h}- F_{k,h}^{(\ell+1)}\bar u_{i,h}\right\|_{a,h}^2
&=& \left\|\bar u_{i,h}- \mathcal P_{k,h}^{(\ell+1)}\bar u_{i,h}\right\|_{a,h}^2 + \left\|(I-F_{k,h}^{\ell})P_{H,h}^{(\ell+1)}\bar u_{i,h}\right\|_{a,h}^2 \\
&\leq& \left(1+\frac{\eta_a^2(V_{H,h}^{(\ell+1)})}{\lambda_{k+1}\big(\delta_{k,i}^{(\ell+1)}\big)^2}\right)
\left\|(I-\mathcal P_{H,h}^{(\ell+1)})\bar u_{i,h}\right\|_{a,h}^2 \\
&\leq&  \left(1+\frac{\eta_a^2(W_H)}{\lambda_{k+1}\big(\delta_{k,i}^{(\ell+1)}\big)^2}\right)
\left\|(I-\mathcal P_{H,h}^{(\ell+1)})\bar u_{i,h}\right\|_{a,h}^2.
\end{eqnarray*}
Thus,
\begin{eqnarray*}
\left\|\bar u_{i,h}- F_{k,h}^{(\ell+1)}\bar u_{i,h}\right\|_{a,h}
&\leq&   \sqrt{1+\frac{\eta_a^2(W_H)}{\lambda_{k+1}\big(\delta_{k,i}^{(\ell+1)}\big)^2}}
\left\|(I-\mathcal P_{H,h}^{(\ell+1)})\bar u_{i,h}\right\|_{a,h}.
\end{eqnarray*}
Together with (\ref{Inequality_18_k}), we arrive at
\begin{equation*}
\left\|\bar u_{i,h} -F_{k,h}^{(\ell+1)}\bar u_{i,h} \right\|_{a,h} \leq
\bar\lambda_{i,h} \sqrt{1+\frac{\eta_a^2(W_H)}{\lambda_{k+1}\big(\delta_{k,i}^{(\ell+1)}\big)^2}}
\left(1+\frac{1}{\lambda_{k+1}\delta_{k,i}^{(\ell)}}\right)\eta_a^2(W_H)\left\|\bar u_{i,h} - F_{k,h}^{(\ell)}\bar u_{i,h} \right\|_{a,h}.
\end{equation*}
We have the following $\left\|\cdot\right\|_b$-error estimate
\begin{equation*}
\left\|\bar u_{i,h}-F_{k,h}^{(\ell+1)}\bar u_{i,h}\right\|_b\leq \left(1+\frac{\mu_{k+1}}{\delta_{k,i}^{(\ell+1)}}\right)
\eta_a(W_H)\left\|\bar u_{i,h}-F_{k,h}^{(\ell+1)}\bar u_{i,h}\right\|_{a,h}.
\end{equation*}
The proof is completed.
\end{proof}

\begin{remark}
The convergence result (\ref{Error_Estimate_Inverse}) in Theorem \ref{thm:algerr_k_NC} means that the augmented subspace methods
have the second order convergence rate. In addition, in order to accelerate the convergence rate, we should
decrease the term $\eta_a(W_H)$ which depends on the coarse space $W_H$.
That is to say, enlarging the subspace $W_H$ can speed up the convergence.
\end{remark}

\begin{remark}\label{Remark_Eigenvalue}
In this paper, we are only concerned with the error estimates for the eigenvector
approximation since the error estimates for the eigenvalue approximation
can be deduced from the following error expansion (refer to (4.9) in \cite{MR3434038}),
\begin{eqnarray*}\label{rayexpan}
0\leq \widehat{\lambda}_i-\bar\lambda_{i,h}
=\frac{a_h(\bar u_{i,h}-\psi,\bar u_{i,h}-\psi)}{b(\psi,\psi)}-\bar\lambda_{i,h}
\frac{b(\bar u_{i,h}-\psi,\bar u_{i,h}-\psi)}{b(\psi,\psi)}+2\frac{a_h(\bar u_{i,h},\psi)-b(\bar \lambda_{i,h}\bar u_{i,h}, \psi)}{b(\psi, \psi)},
\end{eqnarray*}
where $\psi$ is the eigenvector approximation for the exact eigenvector $\bar u_{i,h}$ and
\begin{eqnarray*}
\widehat{\lambda}_i=\frac{a_h(\psi,\psi)}{b(\psi,\psi)}.
\end{eqnarray*}
\end{remark}

It is obvious that the parallel computing method can be used for Step 2 of Algorithm \ref{alg:augk_NC}
since each linear equation can be solved independently. Thereout, the augmented subspace method
can be used to design a type of parallel schemes for eigenvalue problems. Step 3 of Algorithm \ref{alg:augk_NC}
is to solve the eigenvalue problem (\ref{Aug_Eigenvalue_Problem_k}). But in order to assemble the matrices for (\ref{Aug_Eigenvalue_Problem_k}),
we need to do the inner products of the $k$ vectors in the high dimensional space $V_h$, which is a very low scalable
process for the parallel computing \cite{LiXieXuYouZhang,XuXieZhang}.
That is to say, the inner product computation of many high dimensional vectors is indeed a bottleneck for parallel computing.
In order to overcome this essential bottleneck,
we give another version of the augmented subspace method for only one (may be not the smallest one) eigenpair
which represents the single process version of this type of parallel schemes.
The corresponding numerical method is defined by Algorithm \ref{Algorithm_1}.
Here we will also give a sharper error estimate for this type of the method.

In Algorithm \ref{Algorithm_1},  we assume that the given eigenpair
approximation $(\lambda_{i,h}^{(\ell)}, u_{i,h}^{(\ell)})\in\mathbb R\times V_h$  with different superscripts is the closest
to an exact eigenpair $(\bar\lambda_{i,h}, \bar u_{i,h})$ of \eqref{diseigenproblem}.
Based on this setting, we can give the following convergence result for the augmented
subspace method defined by Algorithm \ref{Algorithm_1}.
\begin{algorithm}[hbt!]
\caption{Augmented subspace method for one eigenpair}\label{Algorithm_1}
\begin{enumerate}
\item For $\ell=1$, we define $\widehat u_{i,h}^{(\ell)}=u_{i,h}^{(\ell)}$, and
the augmented subspace $V_{H,h}^{(\ell)} = W_H +{\rm span}\{\widehat u_{i,h}^{(\ell)}\}$.
Then solve the following eigenvalue problem:
Find $(\lambda_{i,h}^{(\ell)},u_{i,h}^{(\ell)})\in \mathbb{R}\times V_{H,h}^{(\ell)}$
such that $a_h(u_{i,h}^{(\ell)},u_{i,h}^{(\ell)})=1$ and
\begin{equation}\label{parallel_correct_eig_exact1}
a_h(u_{i,h}^{(\ell)},v_{H,h}) = \lambda_{i,h}^{(\ell)}b(u_{i,h}^{(\ell)},v_{H,h}),
\ \ \ \ \ \forall v_{H,h}\in V_{H,h}^{(\ell)}.
\end{equation}

\item Solve the following linear boundary value problem:
Find $\widehat{u}_{i,h}^{(\ell+1)}\in V_h$ such that
\begin{equation}\label{Linear_Equation}
a(\widehat{u}_{i,h}^{(\ell+1)},v_h) = \lambda_{i,h}^{(\ell)}b(u_{i,h}^{(\ell)},v_h),
\ \  \forall v_h\in V_h.
\end{equation}
\item Define the augmented subspace $V_{H,h}^{(\ell+1)} = W_H +
{\rm span}\{\widehat{u}_{i,h}^{(\ell+1)}\}$ and solve the following eigenvalue problem:
Find $(\lambda_{i,h}^{(\ell+1)},u_{i,h}^{(\ell+1)})\in \mathbb{R}\times V_{H,h}^{(\ell+1)}$
such that $a_h(u_{i,h}^{(\ell+1)},u_{i,h}^{(\ell+1)})=1$ and
\begin{equation}\label{parallel_correct_eig_exact}
a_h(u_{i,h}^{(\ell+1)},v_{H,h}) = \lambda_{i,h}^{(\ell+1)}b(u_{i,h}^{(\ell+1)},v_{H,h}),
\ \ \ \ \ \forall v_{H,h}\in V_{H,h}^{(\ell+1)}.
\end{equation}
Solve (\ref{parallel_correct_eig_exact}) and the output $(\lambda_{i,h}^{(\ell+1)},u_{i,h}^{(\ell+1)})$
is chosen such that $u_{i,h}^{(\ell+1)}$ has the largest component in ${\rm span}\{\widehat{u}_{i,h}^{(\ell+1)}\}$
among all eigenfunctions of (\ref{parallel_correct_eig_exact}).
\item Set $\ell=\ell+1$ and go to Step 2 for the next iteration until convergence.
\end{enumerate}
\end{algorithm}

\begin{theorem}\label{Theorem_Error_Estimate_1}
For any integer $m\geq 1$, according to the eigenpair approximation $(\lambda_{i,h}^{(m)},u_{i,h}^{(m)})\in\mathbb R\times V_h$,
we define the spectral projector $E_{i,h}^{(m)}: V_h\mapsto {\rm span}\{u_{i,h}^{(m)}\}$ as follows
\begin{eqnarray*}
a_h(E_{i,h}^{(m)}w_h, u_{i,h}^{(m)}) = a_h(w_h, u_{i,h}^{(m)}),\ \  \ \ {\rm for}\  w_h\in V_h.
\end{eqnarray*}
Then the eigenpair approximation $(\lambda_{i,h}^{(\ell+1)},u_{i,h}^{(\ell+1)})\in\mathbb R\times V_h$ produced by
Algorithm \ref{Algorithm_1} satisfies the following error estimates
\begin{align*}
\left\|\bar u_{i,h}-E_{i,h}^{(\ell+1)}\bar u_{i,h}\right\|_{a,h} &\leq \bar\lambda_{i,h} \sqrt{1+\frac{\eta_a^2(W_H)}{\lambda_1\big(\delta_\lambda^{(\ell+1)}\big)^2}}
\left(1+\frac{1}{\lambda_1\delta_\lambda^{(\ell)}}\right)\eta_a^2(W_H)\left\|\bar u_{i,h}-E_{i,h}^{(\ell)}\bar u_{i,h}\right\|_{a,h},\\%\ \ \ \ \ \label{Estimate_h_1_a}\\
\left\|\bar u_{i,h}-E_{i,h}^{(\ell+1)}\bar u_{i,h}\right\|_b &\leq \left(1+\frac{1}{\lambda_1\delta_\lambda^{(\ell+1)}}\right)\eta_a(W_H)
\left\|\bar u_{i,h}-E_{i,h}^{(\ell+1)}\bar u_{i,h}\right\|_{a,h}.%\label{Estimate_h_1_b}
\end{align*}
Here denote by $\delta^{(\ell+1)}_{\lambda}$ the uniform lower bound of the eigenvalue gap $\delta^{(\ell+1)}_{\lambda,h}$, which is defined as
\begin{eqnarray*}
\delta^{(\ell+1)}_{\lambda,h} &:=& \min_{j\neq i} \left|\frac{1}{\lambda^{(\ell+1)}_{j,h}}-\frac{1}{\lambda}\right|.
\end{eqnarray*}
\end{theorem}

\begin{proof}
First, let us consider the error estimate $\left\|\bar u_{i,h}- E_{i,h}^{(\ell)}\bar u_{i,h}\right\|_b$. Because of Algorithm \ref{Algorithm_1},
we understand that $u_{i,h}^{(\ell)}$ may stem from two places, i.e., (\ref{parallel_correct_eig_exact1}) for $\ell=1$ and (\ref{parallel_correct_eig_exact}) for $\ell >1$.
Both cases present the following error estimate for the eigenvector approximation $u_{i,h}^{(\ell)}$. Similarly with the derivation process in the case of the conforming finite element method (refer to Theorem 3.2 in \cite{MR4530254}), we get

\begin{eqnarray}\label{Inequality_131}
\left\|\bar u_{i,h}- E_{i,h}^{(\ell)}\bar u_{i,h}\right\|_b&\leq& \left(1+\frac{1}{\lambda_1\delta_\lambda^{(\ell)}}\right)\eta_a(V_{H,h}^{(\ell)})
\left\|\bar u_{i,h}- E_{i,h}^{(\ell)}\bar u_{i,h}\right\|_{a,h}\nonumber\\
&\leq&\left(1+\frac{1}{\lambda_1\delta_\lambda^{(\ell)}}\right)\eta_a(W_H)
\left\|\bar u_{i,h}- E_{i,h}^{(\ell)}\bar u_{i,h}\right\|_{a,h},
\end{eqnarray}
where we have used the inequality $\eta_a(V_{H,h}^{(\ell)})\leq \eta_a(W_H)$ since $W_H\subset V_{H,h}^{(\ell)}$.

According to the orthogonal property of the projection operator $\mathcal P_{H,h}^{(\ell+1)}$, i.e.,
\begin{equation*}
a_h(\bar u_{i,h}-\mathcal P_{H,h}^{(\ell+1)}\bar u_{i,h}, v_{H,h})=0, \quad \forall v_{H,h}\in V_{H,h}^{(\ell+1)},
\end{equation*}
and the definition of $V_{H,h}^{(\ell+1)}$ in Step 3 of Algorithm \ref{Algorithm_1}, we get
\begin{eqnarray*}
&&\left\|\bar u_{i,h} - \mathcal P_{H,h}^{(\ell+1)}\bar u_{i,h}\right\|_{a,h}^2
= a_h\left(\bar u_{i,h} - \mathcal P_{H,h}^{(\ell+1)}\bar u_{i,h}, \bar u_{i,h} - \mathcal P_{H,h}^{(\ell+1)}\bar u_{i,h}\right)\nonumber\\
&&=a_h\left(\bar u_{i,h}, \bar u_{i,h} - \mathcal P_{H,h}^{(\ell+1)}\bar u_{i,h}\right)\nonumber\\
&&=a_h\left(\bar u_{i,h} -\frac{\bar\lambda_{i,h}}{\lambda_{i,h}^{(\ell)}}q\widehat u_{i,h}^{(\ell+1)}, \bar u_{i,h} - \mathcal P_{H,h}^{(\ell+1)}\bar u_{i,h}\right)\nonumber\\
&&=a_h\left(\bar u_{i,h},  \bar u_{i,h} - \mathcal P_{H,h}^{(\ell+1)}\bar u_{i,h}\right) - \frac{\bar\lambda_{i,h}}{\lambda_{i,h}^{(\ell)}}q
a_h\left(\widehat u_{i,h}^{(\ell+1)}, \bar u_{i,h} - \mathcal P_{H,h}^{(\ell+1)}\bar u_{i,h}\right).
\end{eqnarray*}
And by virtue of \eqref{diseigenproblem}, (\ref{Linear_Equation}) and $V_{H,h}^{(\ell+1)}\subset V_h$, we have
\begin{equation*}
\left\|\bar u_{i,h} - \mathcal P_{H,h}^{(\ell+1)}\bar u_{i,h}\right\|_{a,h}^2 = \bar\lambda_{i,h} b\left(\bar u_{i,h},  \bar u_{i,h} - \mathcal P_{H,h}^{(\ell+1)}\bar u_{i,h}\right) - \bar\lambda_{i,h}
b\left(q u_{i,h}^{(\ell)}, \bar u_{i,h} - \mathcal P_{H,h}^{(\ell+1)}\bar u_{i,h}\right).
\end{equation*}
Since the spectral projection $E_{i,h}^{(\ell)}$ satisfies
\begin{equation*}
E_{i,h}^{(\ell)}\bar u_{i,h} = q u_{i,h}^{(\ell)},
\end{equation*}
combined with \eqref{Inequality_131}, \eqref{eq:NiCai_NC} and the inequality $\eta_a(V_{H,h}^{(\ell+1)})\leq \eta_a(W_H)$, we obtain
\begin{eqnarray}\label{Inequality_16_2}
&&\left\|\bar u_{i,h} - \mathcal P_{H,h}^{(\ell+1)}\bar u_{i,h}\right\|_{a,h}^2
= \bar\lambda_{i,h} b\left(\bar u_{i,h} - E_{i,h}^{(\ell)}\bar u_{i,h}, \bar u_{i,h} - \mathcal P_{H,h}^{(\ell+1)}\bar u_{i,h}\right)\nonumber\\
&&\leq \bar\lambda_{i,h}\left\|\bar u_{i,h} - E_{i,h}^{(\ell)}\bar u_{i,h}\right\|_b\left\|\bar u_{i,h} - \mathcal P_{H,h}^{(\ell+1)}\bar u_{i,h}\right\|_b\nonumber\\
&&\leq  \bar\lambda_{i,h}\left(1+\frac{1}{\lambda_1\delta_\lambda^{(\ell)}}\right) \eta_a(W_H)\left\|\bar u_{i,h}- E_{i,h}^{(\ell)}\bar u_{i,h}\right\|_{a,h}
\eta_a(V_{H,h}^{(\ell+1)})\left\|\bar u_{i,h} - \mathcal P_{H,h}^{(\ell+1)}\bar u_{i,h}\right\|_{a,h}\nonumber\\
&&\leq  \bar\lambda_{i,h} \left(1+\frac{1}{\lambda_1\delta_\lambda^{(\ell)}}\right)
\eta_a^2(W_H)\left\|\bar u_{i,h}- E_{i,h}^{(\ell)}\bar u_{i,h}\right\|_{a,h}\left\|\bar u_{i,h} - \mathcal P_{H,h}^{(\ell+1)}\bar u_{i,h}\right\|_{a,h}.
\end{eqnarray}
Since the approximation $u_{i,h}^{(\ell+1)}$ only comes from (\ref{parallel_correct_eig_exact}), we have
\begin{eqnarray*}
\left\|\bar u_{i,h}- E_{i,h}^{(\ell+1)}\bar u_{i,h}\right\|_{a,h}
&\leq& \sqrt{1+\frac{\eta_a^2(V_{H,h}^{(\ell+1)})}{\lambda_1\big(\delta_\lambda^{(\ell+1)}\big)^2}}
\left\|(I-\mathcal P_{H,h}^{(\ell+1)})\bar u_{i,h}\right\|_{a,h}\nonumber\\
&\leq&  \sqrt{1+\frac{\eta_a^2(W_H)}{\lambda_1\big(\delta_\lambda^{(\ell+1)}\big)^2}}
\left\|(I-\mathcal P_{H,h}^{(\ell+1)})\bar u_{i,h}\right\|_{a,h}.
\end{eqnarray*}

From (\ref{Inequality_16_2}), there holds
\begin{eqnarray}\label{Inequality_17_2}
\left\|\bar u_{i,h} - \mathcal P_{H,h}^{(\ell+1)}\bar u_{i,h}\right\|_{a,h} \leq \bar\lambda_{i,h} \left(1+\frac{1}{\lambda_1\delta_\lambda^{(\ell)}}\right)\eta_a^2(W_H)\left\|\bar u_{i,h}- E_{i,h}^{(\ell)}\bar u_{i,h}\right\|_{a,h}.
\end{eqnarray}
And considering (\ref{Inequality_17_2}), we have the following estimate
\begin{eqnarray}\label{Inequality_18}
\left\|\bar u_{i,h}-E_{i,h}^{(\ell+1)}\bar u_{i,h}\right\|_{a,h} \leq \bar\lambda_{i,h}
\sqrt{1+\frac{\eta_a^2(W_H)}{\lambda_1\big(\delta_\lambda^{(\ell+1)} \big)^2}}\left(1+\frac{1}{\lambda_1\delta_\lambda^{(\ell)}}\right) \eta_a^2(W_H)\left\|\bar u_{i,h}-E_{i,h}^{(\ell)}\bar u_{i,h}\right\|_{a,h}.
\end{eqnarray}
The following $\left\|\cdot\right\|_b$-error estimate holds
\begin{eqnarray}\label{Inequality_19_2}
&&\left\|\bar u_{i,h}- E_{i,h}^{(\ell+1)}\bar u_{i,h}\right\|_b\leq \left(1+\frac{1}{\lambda_1\delta_\lambda^{(\ell+1)}}\right) \eta_a(W_H)\left\|\bar u_{i,h}- E_{i,h}^{(\ell+1)}\bar u_{i,h}\right\|_{a,h}.
\end{eqnarray}
From (\ref{Inequality_18}) and (\ref{Inequality_19_2}), the proof is completed.
\end{proof}

\begin{corollary}
Under the conditions of Theorem \ref{Theorem_Error_Estimate_1}, the eigenfunction approximation $u_{i,h}^{(\ell+1)}$ has the following error estimates
\begin{eqnarray}
\left\|\bar u_{i,h}-E_{i,h}^{(\ell+1)}\bar u_{i,h}\right\|_{a,h} &\leq&\big(\gamma(\bar\lambda_{i,h})\big)^\ell\  \left\|\bar u_{i,h}-E_{i,h}^{(1)}\bar u_{i,h}\right\|_{a,h},\label{Estimate_h_1_a_ell}\\
\left\|\bar u_{i,h}-E_{i,h}^{(\ell+1)}\bar u_{i,h}\right\|_b&\leq& \left(1+\frac{1}{\lambda_1\delta_\lambda^{(\ell+1)}}\right)\eta_a(W_H)
\left\|\bar u_{i,h}-E_{i,h}^{(\ell+1)}\bar u_{i,h}\right\|_{a,h},\label{Estimate_h_1_b_ell}
\end{eqnarray}
where
\begin{eqnarray}\label{Definition_Gamma}
\gamma(\bar\lambda_{i,h}) = \bar\lambda_{i,h} \sqrt{1+\frac{\eta_a^2(W_H)}{\lambda_1\big(\delta_\lambda^{(\ell+1)} \big)^2}}
\left(1+\frac{1}{\lambda_1\delta_\lambda^{(\ell)}}\right)\eta_a^2(W_H).
\end{eqnarray}
\end{corollary}
The error estimate for the eigenvalue approximation $\lambda_{i,h}^{(\ell)}$ can be deduced from
Theorem \ref{Theorem_Error_Estimate_1}  and Remark \ref{Remark_Eigenvalue}.

\section{Numerical experiments}\label{Section_4}
In this section, numerical experiments are presented to validate our theoretical results. Here, we are concerned with the Laplace eigenvalue problem \eqref{originproblem}, where the computing domain is set to be the unit square $\Omega=(0,1)\times (0,1)$. $W_H$ and $V_h$ are chosen as the conforming linear element and the CR element spaces defined on the coarse mesh $\mathcal T_H$ and the fine mesh $\mathcal T_h$, respectively. The fine mesh $\mathcal T_h$ is obtained from the coarse mesh $\mathcal T_H$ by the regular refinement of uniform triangular mesh.

Since the coarse space $W_H$ is the conforming linear finite element space defined on the coarse mesh $\mathcal T_H$, together with the theories of the error estimates of CR element and conforming linear element \cite{MR2373954, MR0520174}, it is known that the following estimate holds when the mesh size $h < H$,
\begin{align*}
\eta_a(W_H)&\le \sup_{f\in L^2(\Omega),\|f\|_b=1}\left\{\|T f-T_h f\|_{a,h} + \inf_{v_{H}\in W_{H}}\|Tf-v_{H}\|_{a}\right\}\\
&\le CH,
\end{align*}
where the constant $C$ depends on shape of the meshes $\mathcal T_H$ and $\mathcal T_h$.

Based on Theorems \ref{thm:algerr_k_NC} and \ref{Theorem_Error_Estimate_1},
the convergence results can be concluded with the following inequalities
\begin{eqnarray}
&&\left\|\bar u_{i,h} -F_{k,h}^{(\ell+1)}\bar u_{i,h}\right\|_{a,h} \leq C\big(CH\big)^{2\ell}\left\|\bar u_{i,h} - F_{k,h}^{(1)}\bar u_{i,h}\right\|_{a,h},\ \ \ \ i=1, \cdots, k,\label{Test_1_1}\\
&&\left\|\bar u_{i,h} -F_{k,h}^{(\ell+1)}\bar u_{i,h} \right\|_b\leq CH \left\|\bar u_{i,h} -F_{k,h}^{(\ell+1)}\bar u_{i,h}\right\|_{a,h},\ \ \ \ i=1, \cdots, k,\label{Test_1_0}
\end{eqnarray}
and
\begin{eqnarray}
\left\|\bar u_{i,h}-E_{i,h}^{(\ell+1)}\bar u_{i,h}\right\|_{a,h} &\leq& C\big(CH\big)^{2\ell}\left\|\bar u_{i,h}-E_{i,h}^{(1)}\bar u_{i,h}\right\|_{a,h},\label{Test_2_1}\\
\left\|\bar u_{i,h}-E_{i,h}^{(\ell+1)}\bar u_{i,h}\right\|_b&\leq& CH\left\|\bar u_{i,h}-E_{i,h}^{(\ell+1)}\bar u_{i,h}\right\|_{a,h}.\label{Test_2_0}
\end{eqnarray}
One of the purposes of this section is to check the overall error estimates and algebraic error estimates \eqref{Test_1_1}-\eqref{Test_2_0} of our proposed methods. It should be noted that the exact finite element eigenfunction is obtained by solving the eigenvalue
problem directly on the fine space $V_h$, which is the CR element space defined on the fine mesh $\mathcal T_h$. To make it more intuitive, in all the following figures, the notations with
and without ``{\tt dir}" superscript stand for the exact finite element eigenfunctions and the augmented subspace approximations, respectively. The other purpose of this section is to carry out the numerical tests for the computational complexity by comparing Algorithms \ref{alg:augk_NC} and \ref{Algorithm_1} with the Krylov-Schur method directly applied on the final mesh without coarsening to verify the advantage of our algorithms.

We should note that all the following numerical tests are accomplished on LSSC-IV in the State Key Laboratory of Scientific and Engineering Computing, Academy of Mathematics and Systems Science, Chinese Academy of Sciences, where each computing node has two $18$-core Intel Xeon Gold $6140$ processors at $2.3$ GHz and $192$ GB memory.
The linear equations \eqref{eq:Linear_Equation_k_NC} in Algorithm \ref{alg:augk_NC} and (\ref{Linear_Equation}) in Algorithm \ref{Algorithm_1} are solved by the package PETSc \cite{petsc-web-page,petsc-user-ref,petsc-efficient} with the geometric multigrid method. The eigenvalue problems (\ref{EigenProblem_k_NC_1}) and (\ref{Aug_Eigenvalue_Problem_k}) in Algorithm \ref{alg:augk_NC} and (\ref{parallel_correct_eig_exact1}) and (\ref{parallel_correct_eig_exact}) in Algorithm \ref{Algorithm_1} are solved by the Krylov-Schur algorithm from SLEPc \cite{SLEPc}. The single processor is adopted for the convergence tests and $36$ processors for the computational complexity tests.

\subsection{Tests for overall error estimates and computational efficiency}
In this subsection, we carry out numerical examples to check the overall error estimates and the computational complexity of our algorithms. In the tests of error and computational efficiency, we fix the coarse mesh sizes $H=\sqrt{2}/8$ and $H=\sqrt{2}/32$, respectively, and divide the fine mesh $\mathcal{T}_h$. The initial eigenfunction approximations are obtained in two steps: (1) The initial coarse eigenfunction approximations are produced by solving the eigenvalue problem (\ref{originproblem}) on the coarse space $V_H$; (2) The interpolation matrix is used to project the initial coarse eigenfunction approximations onto the space $V_h$ to get the initial eigenfunction approximations $u_{1,h}^{(1)},\cdots, u_{k,h}^{(1)}$. Then we do the iteration steps by the augmented subspace
method defined by Algorithms \ref{alg:augk_NC} and \ref{Algorithm_1}.

For comparison, we conduct the tests that the SLEPc solver (Krylov-Schur method) is applied directly on the mesh $\mathcal{T}_h$ without coarsening, which we also call the single level solver. Figure \ref{overall} gives the overall errors for the first $4$ eigenvalues $2\pi^2$, $5\pi^2$, $5\pi^2$, $8\pi^2$ and their corresponding eigenfunctions by our algorithms and the single level solver with $H=\sqrt{2}/8$, and Figure \ref{CPUs} shows the CPU time for computing the first and the smallest $4$ eigenpair approximations with $H=\sqrt{2}/32$.

From Figure \ref{overall}, it follows that the error convergence orders of the first $4$ eigenvalues and their corresponding eigenfunctions by our algorithms are $O(h^2)$ and $O(h)$, respectively. Furthermore, from the left subfigure of Figure \ref{overall}, we can find that the augmented subspace method can also obtain the lower bound approximations of the eigenvalues.

As for the computational efficiency, from Figures \ref{overall} and \ref{CPUs}, it can be seen that the augmented subspace method provides almost the same results as the single level solver but with smaller computational work.

\begin{figure}[!hbt]
\centering
\includegraphics[width=6cm,height=4.5cm]{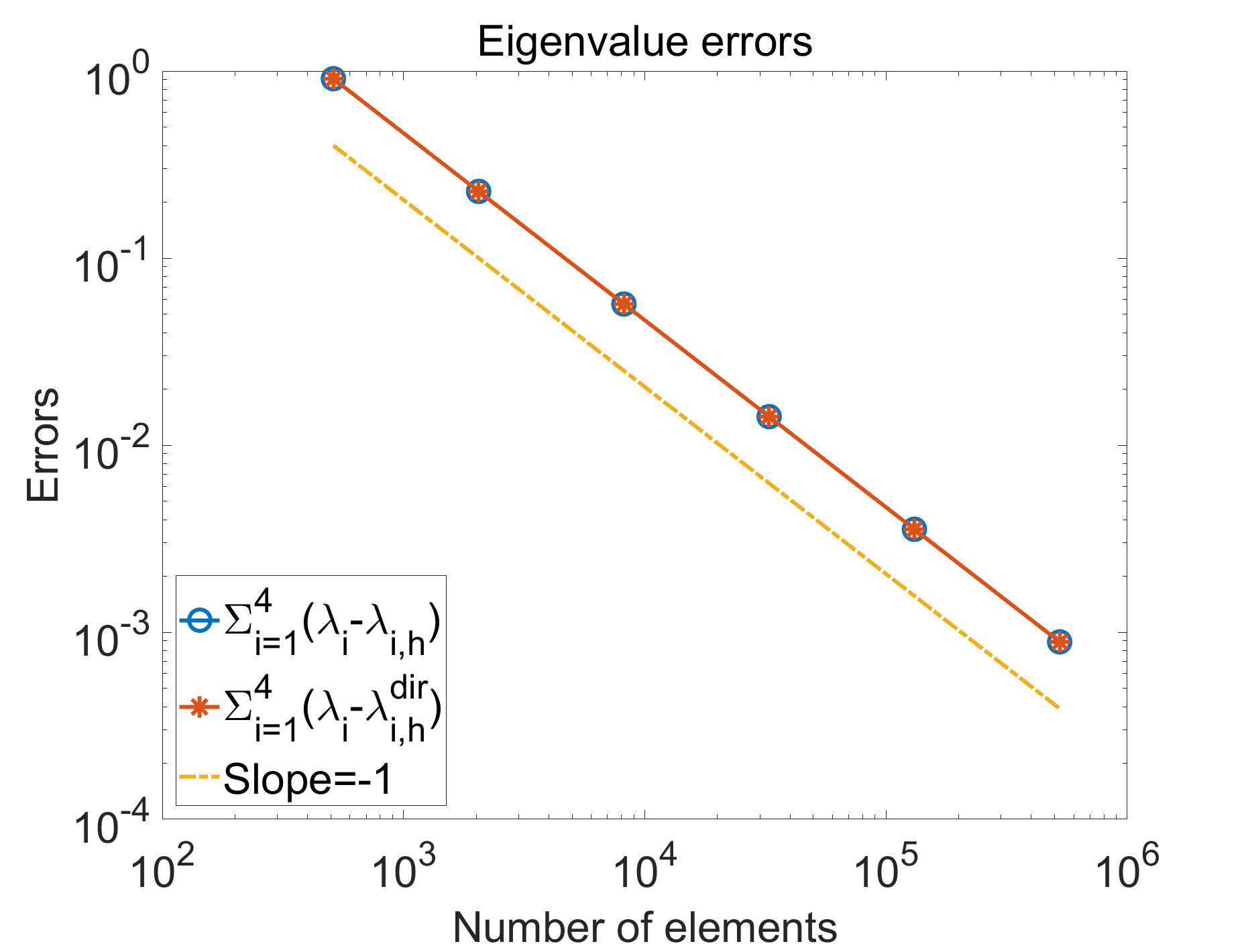}
\includegraphics[width=6cm,height=4.5cm]{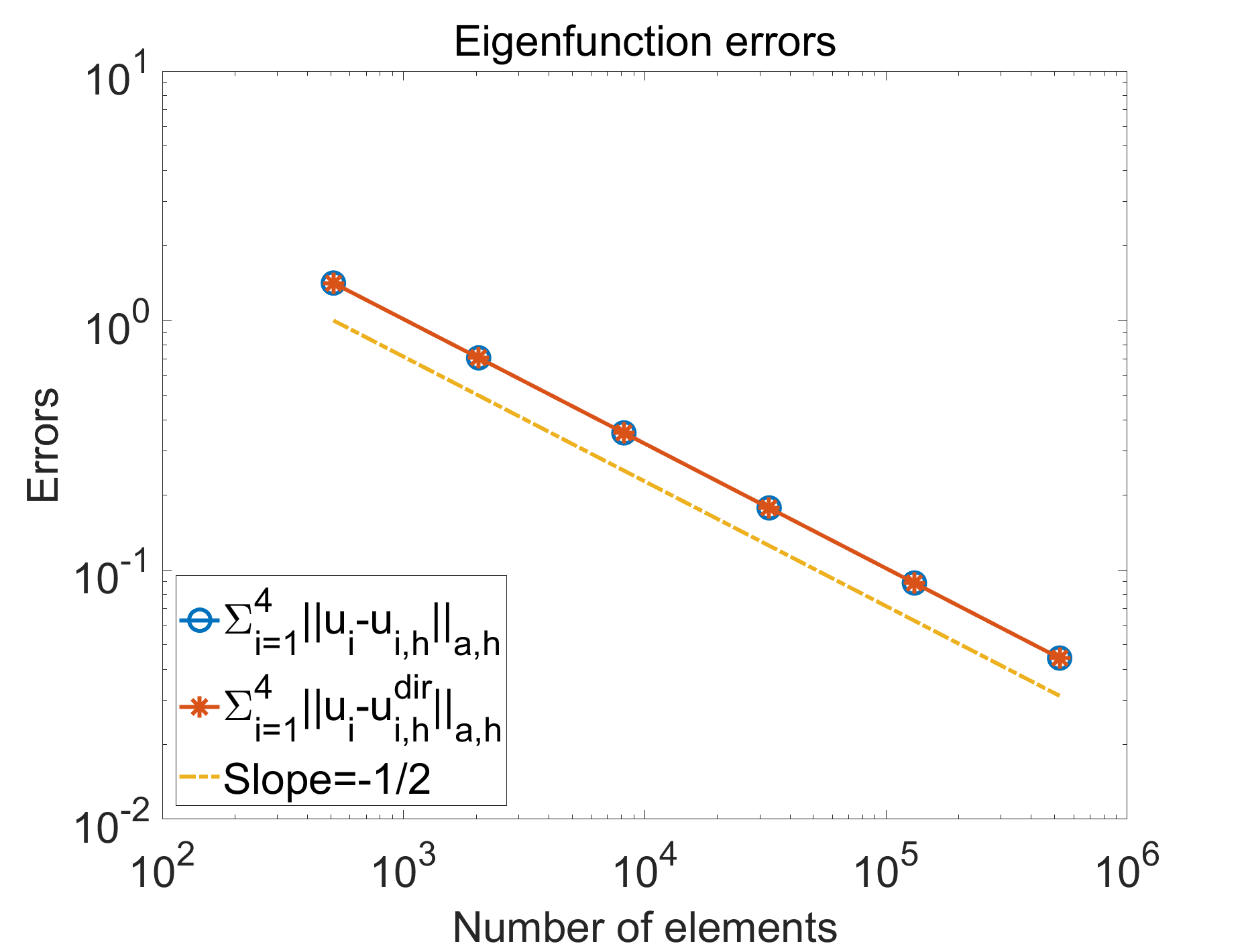}
\caption{Errors for the eigenpair approximations by our algorithms and the single level solver for the first $4$ eigenvalues $2\pi^2$, $5\pi^2$, $5\pi^2$, $8\pi^2$ and their corresponding eigenfunctions with $H=\sqrt{2}/8$.}
\label{overall}
\end{figure}

\begin{figure}[!hbt]
\centering
\includegraphics[width=6cm,height=4.5cm]{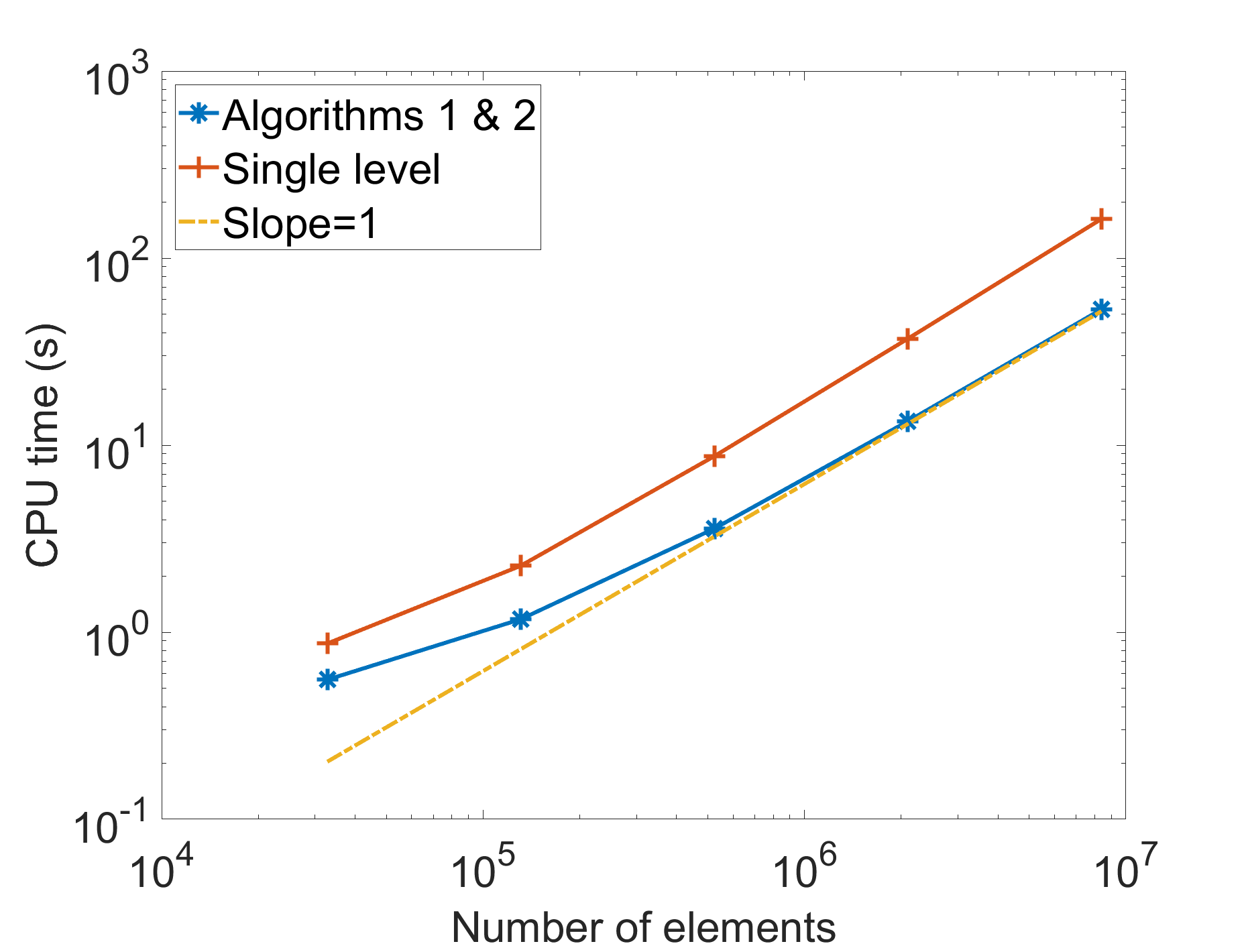}
\includegraphics[width=6cm,height=4.5cm]{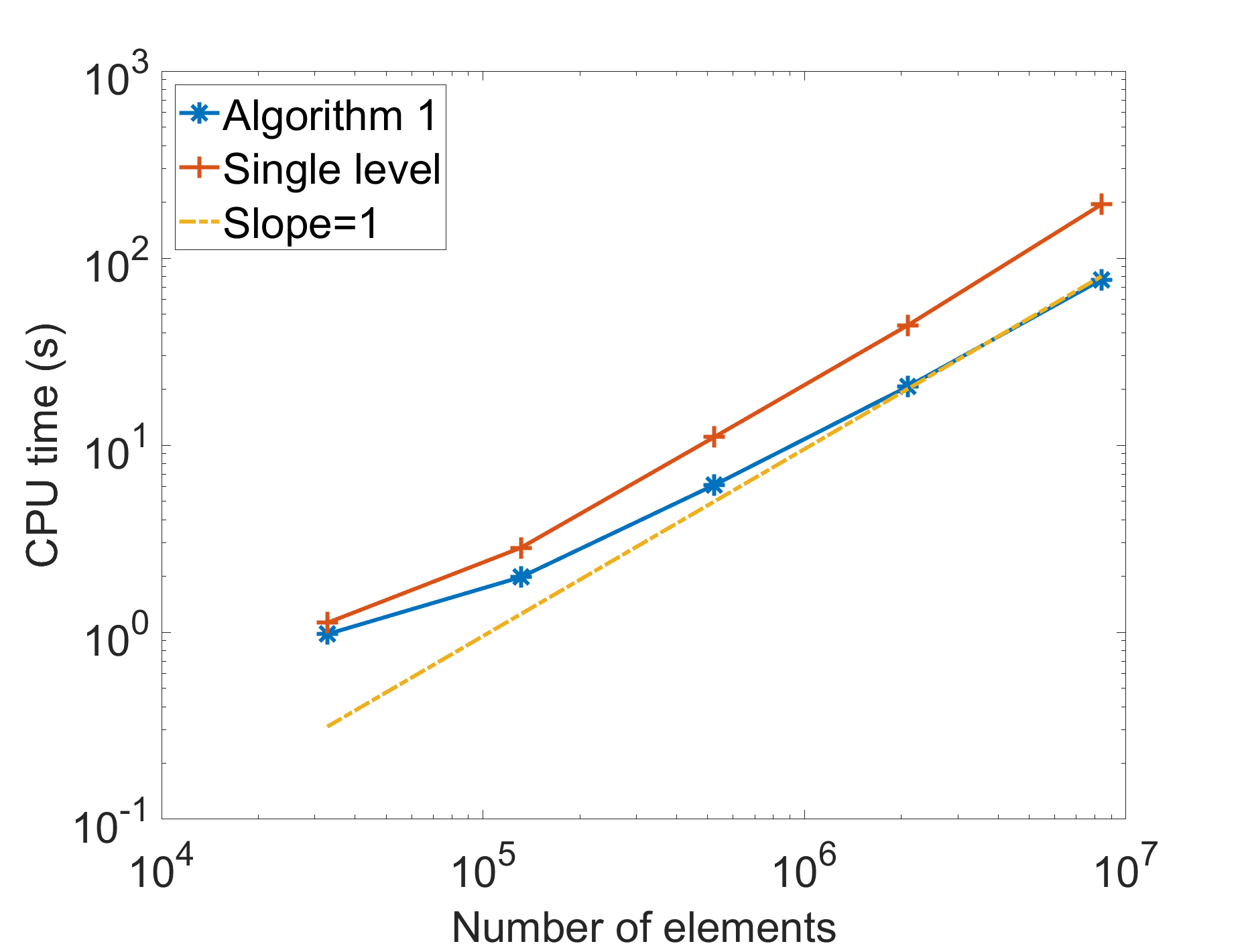}
\caption{CPU time for our algorithms and the single level solver with $H=\sqrt{2}/32$: The left subfigure shows the CPU time for the first eigenpair approximation  and the right subfigure shows the CPU time for the smallest $4$ eigenpair approximations.}
\label{CPUs}
\end{figure}

\subsection{Tests for algebraic error estimates}
In this subsection, we check the algebraic error estimates of our proposed algorithms, which is also one of the main contributions of this paper. In the following numerical example, we set the fine mesh size $h=\sqrt{2}/512$ for testing the convergence. The initial eigenfunction approximation is produced by solving the eigenvalue
problem \eqref{originproblem} on the coarse space $W_H$. Then we do the iteration steps by the augmented subspace
method defined by Algorithms \ref{alg:augk_NC} and \ref{Algorithm_1}.

In order to validate the convergence results stated in (\ref{Test_1_1})-(\ref{Test_2_0}),
we check the numerical errors corresponding to the linear finite element space $W_H$
with different sizes $H$.  The aim is to check the dependence of the convergence rate
on the mesh size $H$.

Figure \ref{Result_Coarse_Mesh} shows the convergence behaviors for the first eigenfunction by
the augmented subspace methods corresponding to the coarse mesh size $H=\sqrt{2}/8$, $\sqrt{2}/16$, $\sqrt{2}/32$ and $\sqrt{2}/64$, respectively.
The convergence rates related with $\|\cdot\|_{a,h}$ and $\|\cdot\|_b$ are $0.048503$, $0.013692$, $0.0035433$ and $0.00089758$, and $0.052093$, $0.014819$, $0.0038451$ and $0.001132$, separately.
These results show that the augmented subspace method defined by Algorithms \ref{alg:augk_NC} and \ref{Algorithm_1} have the second order convergence speed
which also validates the results (\ref{Test_1_1})-(\ref{Test_2_0}).
\begin{figure}[http!]
\centering
\includegraphics[width=6cm,height=4.5cm]{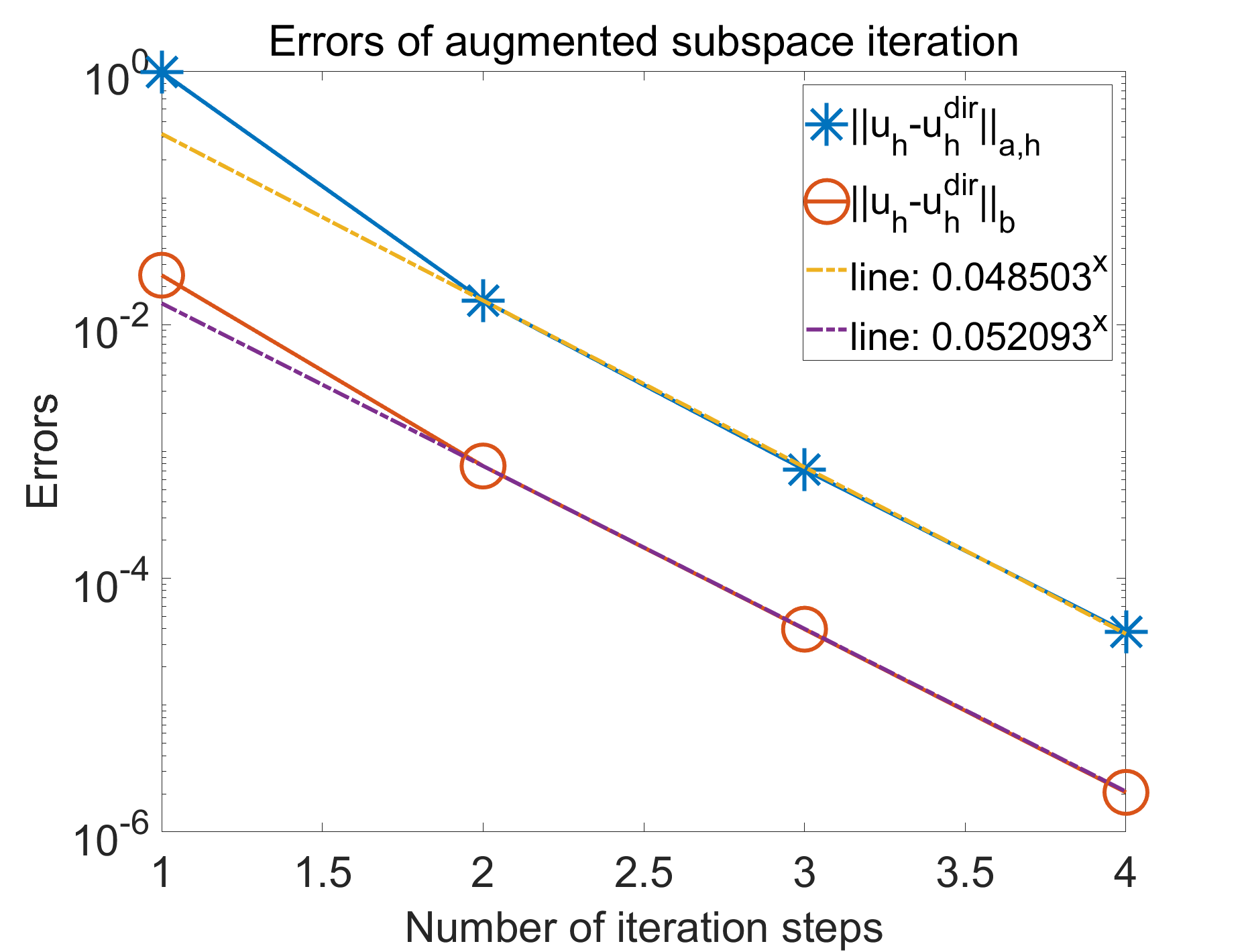}
\includegraphics[width=6cm,height=4.5cm]{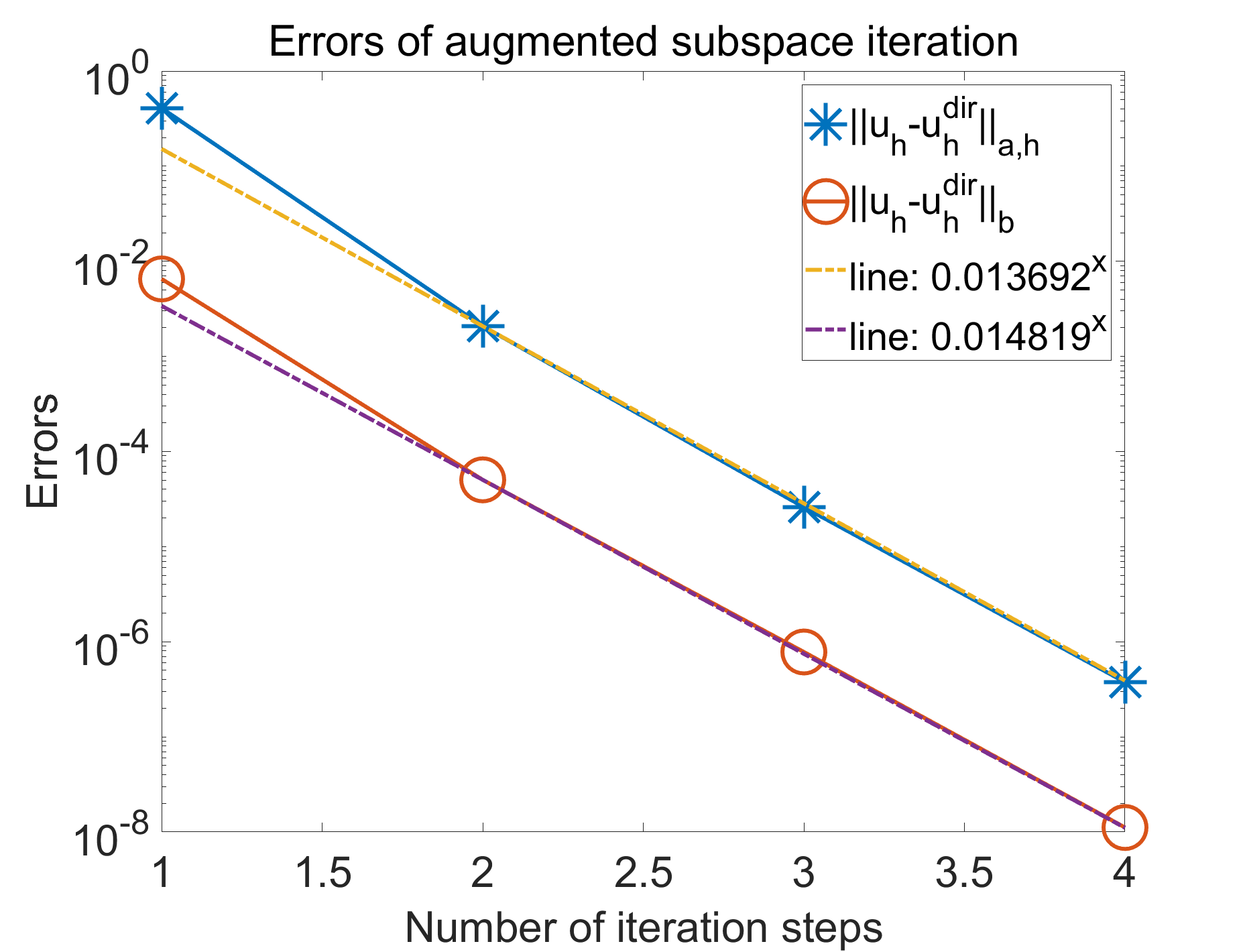}\\
\includegraphics[width=6cm,height=4.5cm]{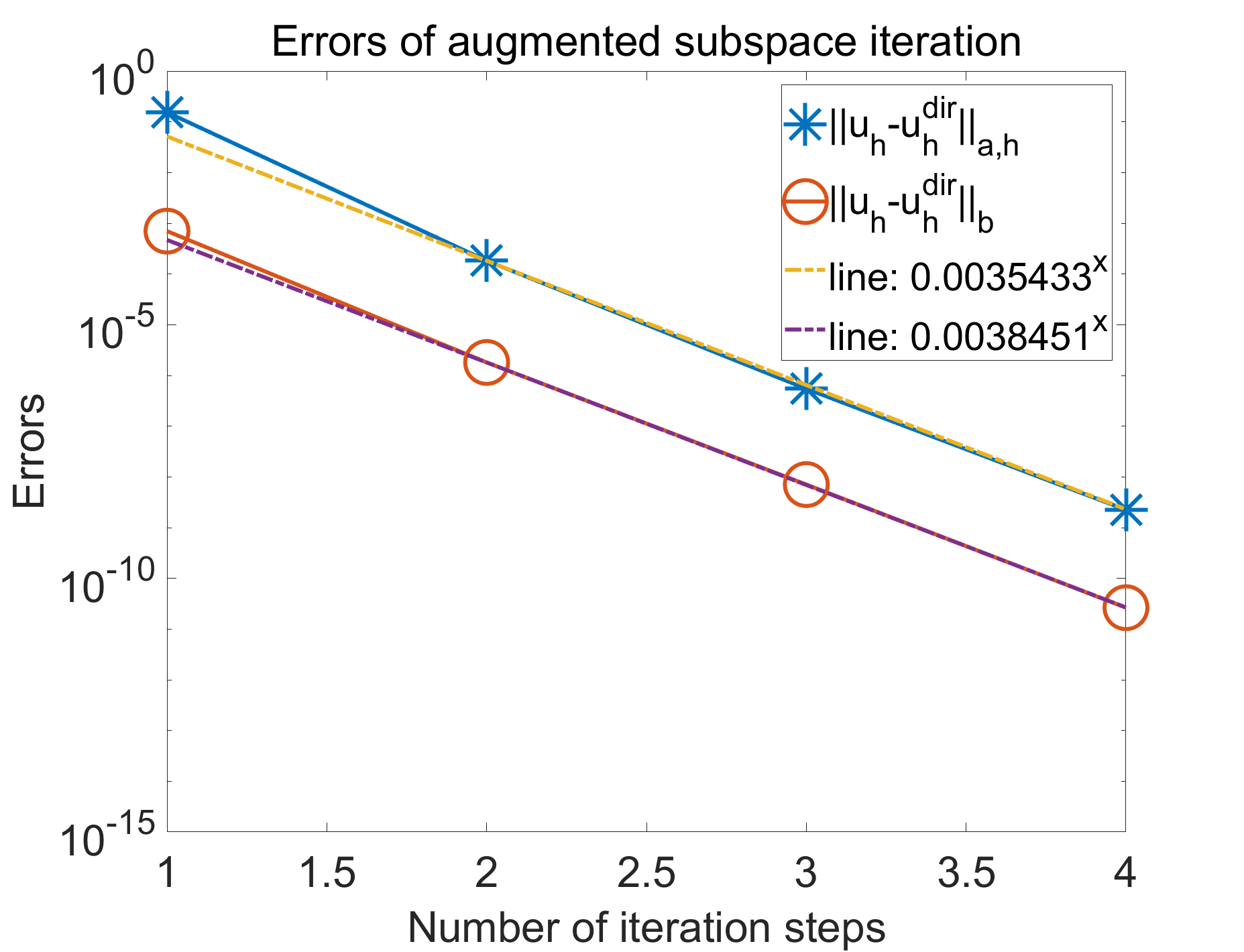}
\includegraphics[width=6cm,height=4.5cm]{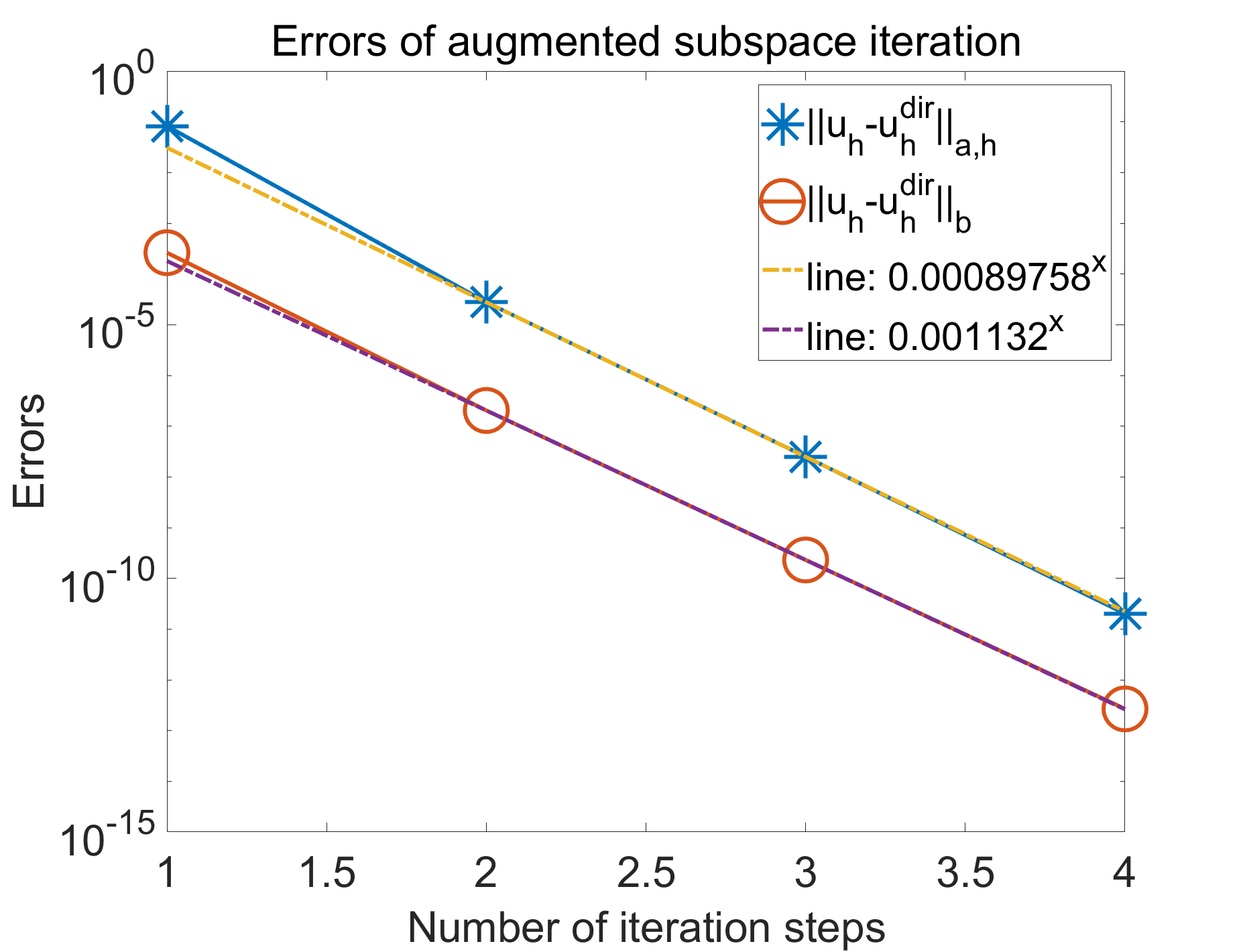}
\caption{The convergence behaviors for the first eigenfunction by Algorithm \ref{alg:augk_NC}
corresponding to the coarse mesh size $H=\sqrt{2}/8$, $\sqrt{2}/16$, $\sqrt{2}/32$ and $\sqrt{2}/64$, respectively.}\label{Result_Coarse_Mesh}
\end{figure}

Then, we check the performance of Algorithm \ref{alg:augk_NC} for computing the smallest $4$ eigenpairs.
Figure \ref{Result_Coarse_Mesh_4} shows the corresponding convergence behaviors  for the smallest $4$ eigenfunctions
by Algorithm \ref{alg:augk_NC} with the coarse space being the linear finite element space on the mesh with size $H=\sqrt{2}/8$, $\sqrt{2}/16$, $\sqrt{2}/32$ and $\sqrt{2}/64$, respectively.
Taking the $4$-th eigenfunction for example, we can find that the corresponding convergence rates are $0.52894$, $0.14163$, $0.033235$ and $0.0084715$,
which states the second order convergence speed of the method defined by Algorithm \ref{alg:augk_NC}.
Furthermore, from Figure \ref{Result_Coarse_Mesh_4}, we can find
the convergence rate for the $4$-th eigenfucntion is slower than that for the $1$-st eigenfunction which
is consistent with Theorem \ref{thm:algerr_k_NC}.
\begin{figure}[http!]
\centering
\includegraphics[width=6cm,height=4.5cm]{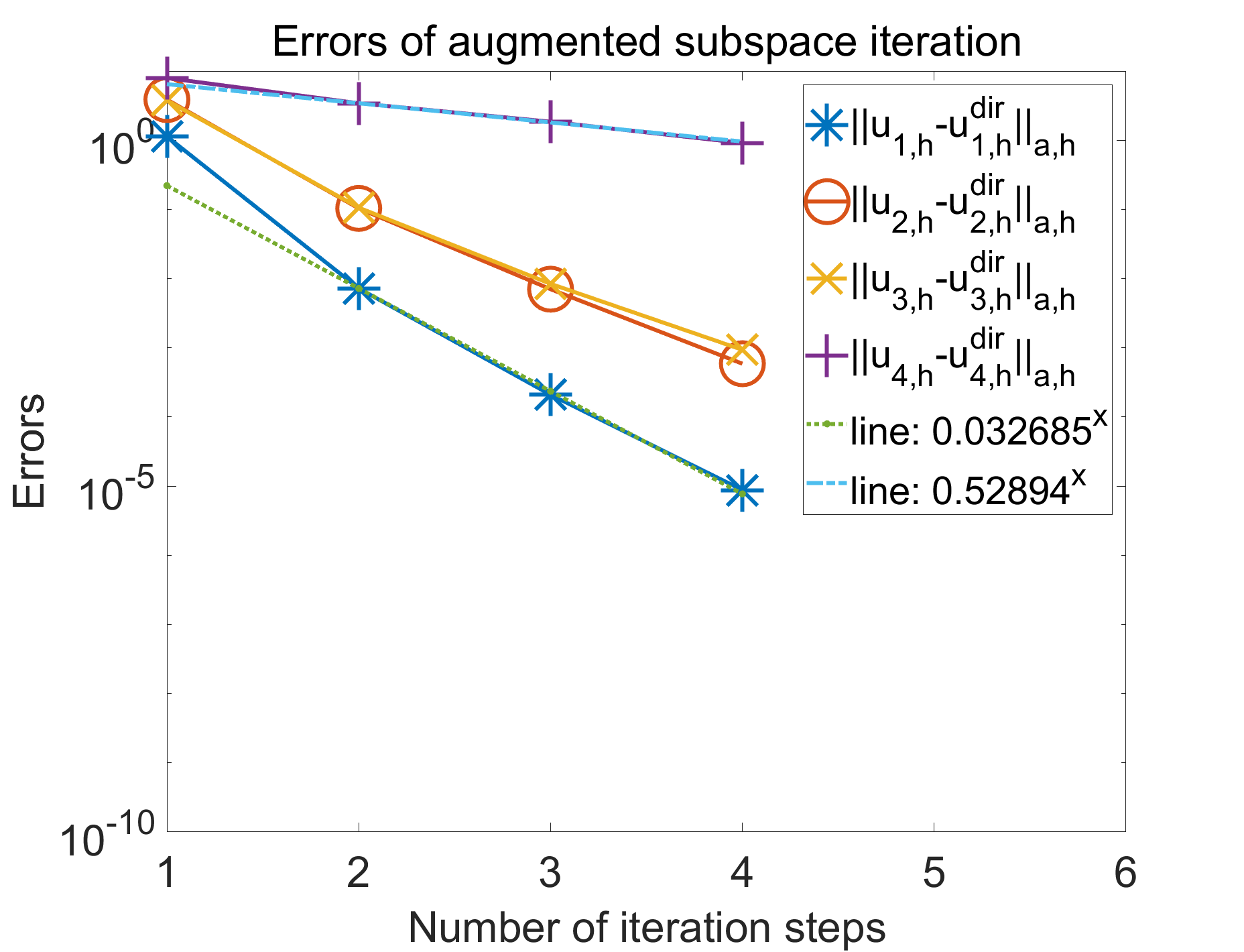}
\includegraphics[width=6cm,height=4.5cm]{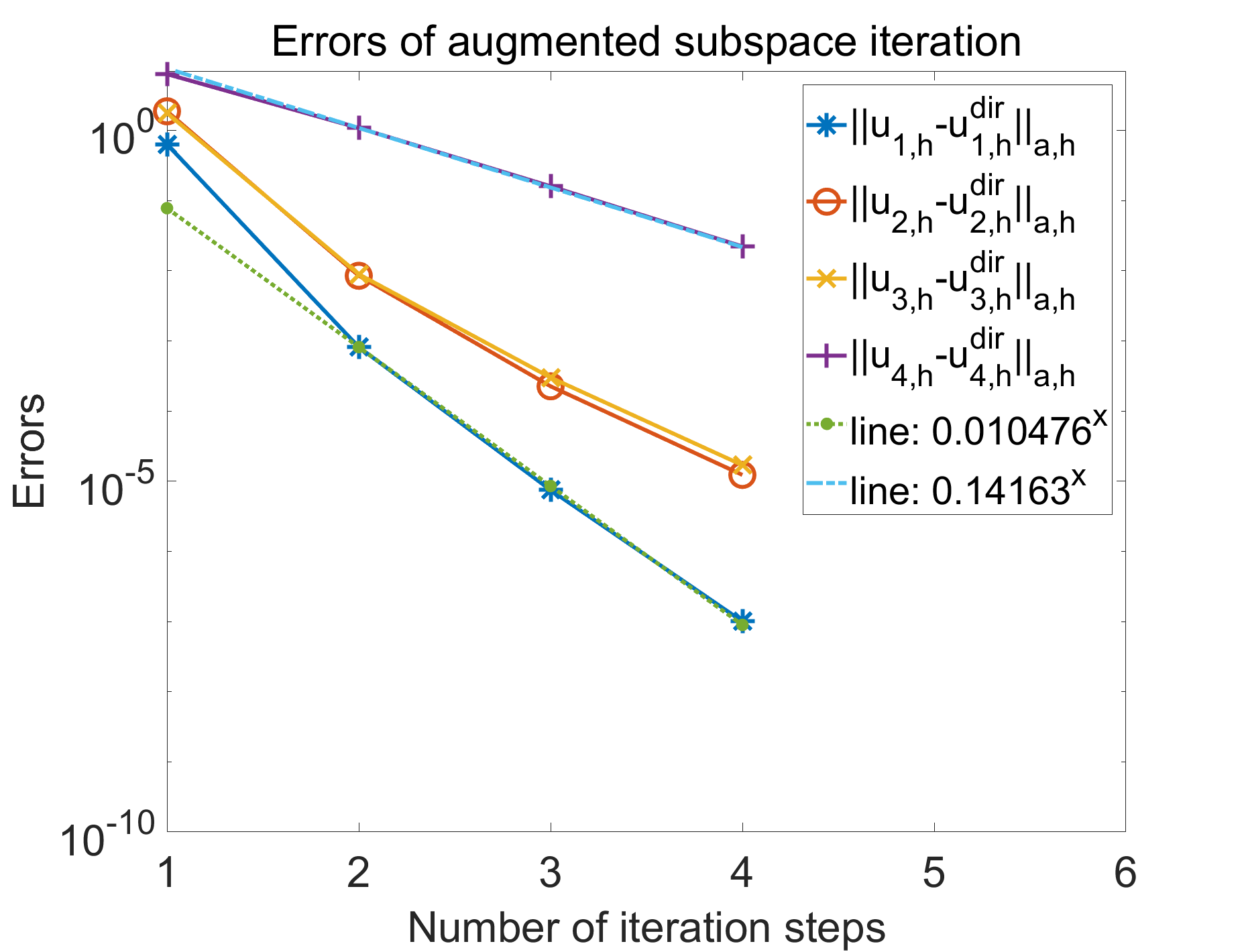}\\
\includegraphics[width=6cm,height=4.5cm]{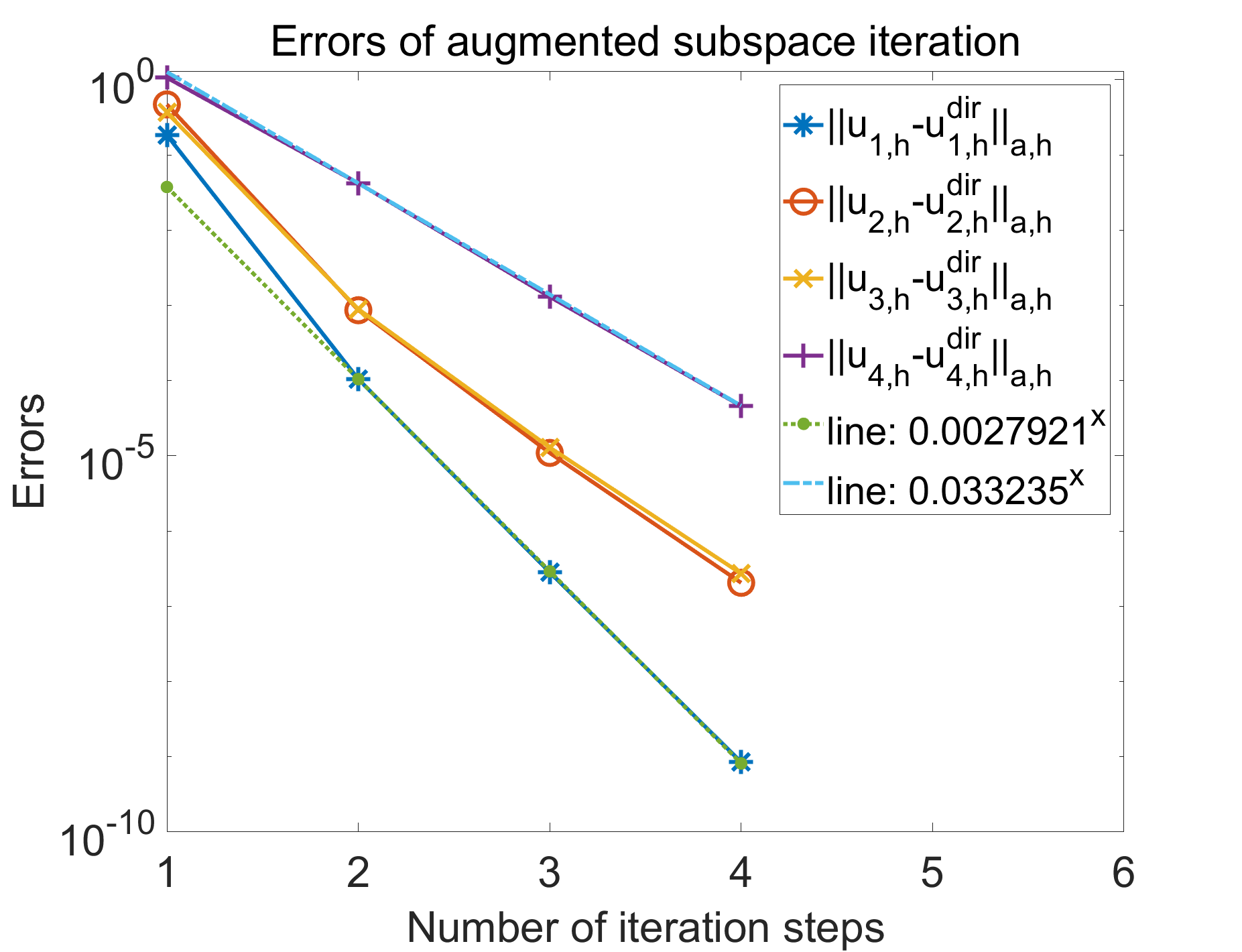}
\includegraphics[width=6cm,height=4.5cm]{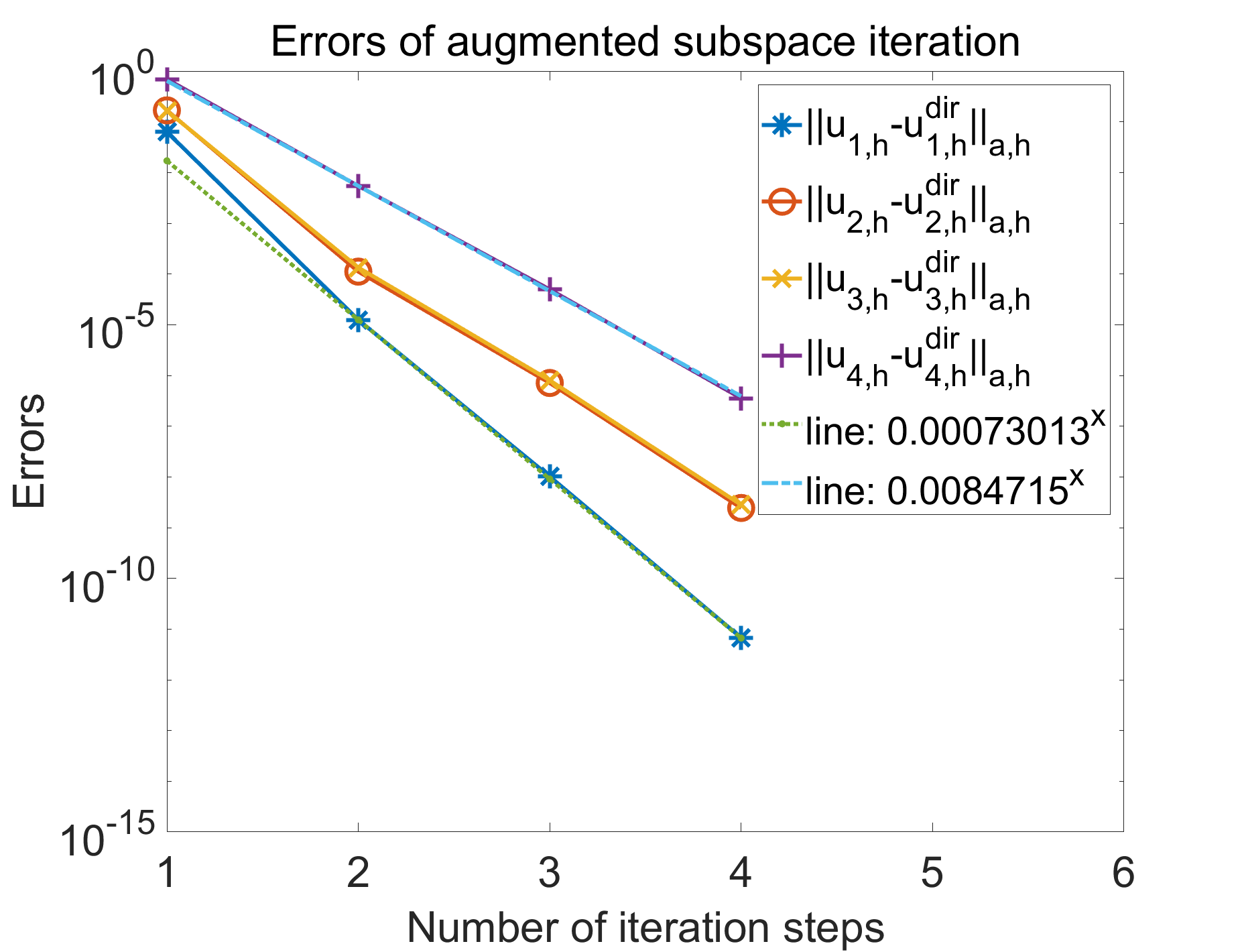}
\caption{The convergence behaviors for the smallest $4$ eigenfunctions by Algorithm \ref{alg:augk_NC}
with the coarse space being the linear finite element space on the mesh with size $H=\sqrt{2}/8$, $\sqrt{2}/16$, $\sqrt{2}/32$ and $\sqrt{2}/64$, respectively.}\label{Result_Coarse_Mesh_4}
\end{figure}

The final task is to check the performance of Algorithm \ref{Algorithm_1} for computing the only $4$-th eigenpair.
Figure \ref{Result_Coarse_Mesh_4_Only} shows the corresponding convergence behaviors  for the only $4$-th eigenfunction
by Algorithm \ref{Algorithm_1} with the coarse space being the linear finite element space on
the mesh with size $H=\sqrt{2}/8$, $\sqrt{2}/16$, $\sqrt{2}/32$ and $\sqrt{2}/64$, respectively. The convergence rates corresponding to $\|\cdot\|_{a,h}$ and $\|\cdot\|_b$ shown in
Figure \ref{Result_Coarse_Mesh_4_Only} are $0.5321$, $0.14456$, $0.033682$ and $0.0085807$, and $0.35757$, $0.12384$, $0.034523$ and $0.0089362$, separately.
These results show that the augmented subspace method defined by
Algorithm \ref{Algorithm_1} has the second order convergence speed which validates the results (\ref{Test_2_1})-(\ref{Test_2_0}).
\begin{figure}[http!]
\centering
\includegraphics[width=6cm,height=4.5cm]{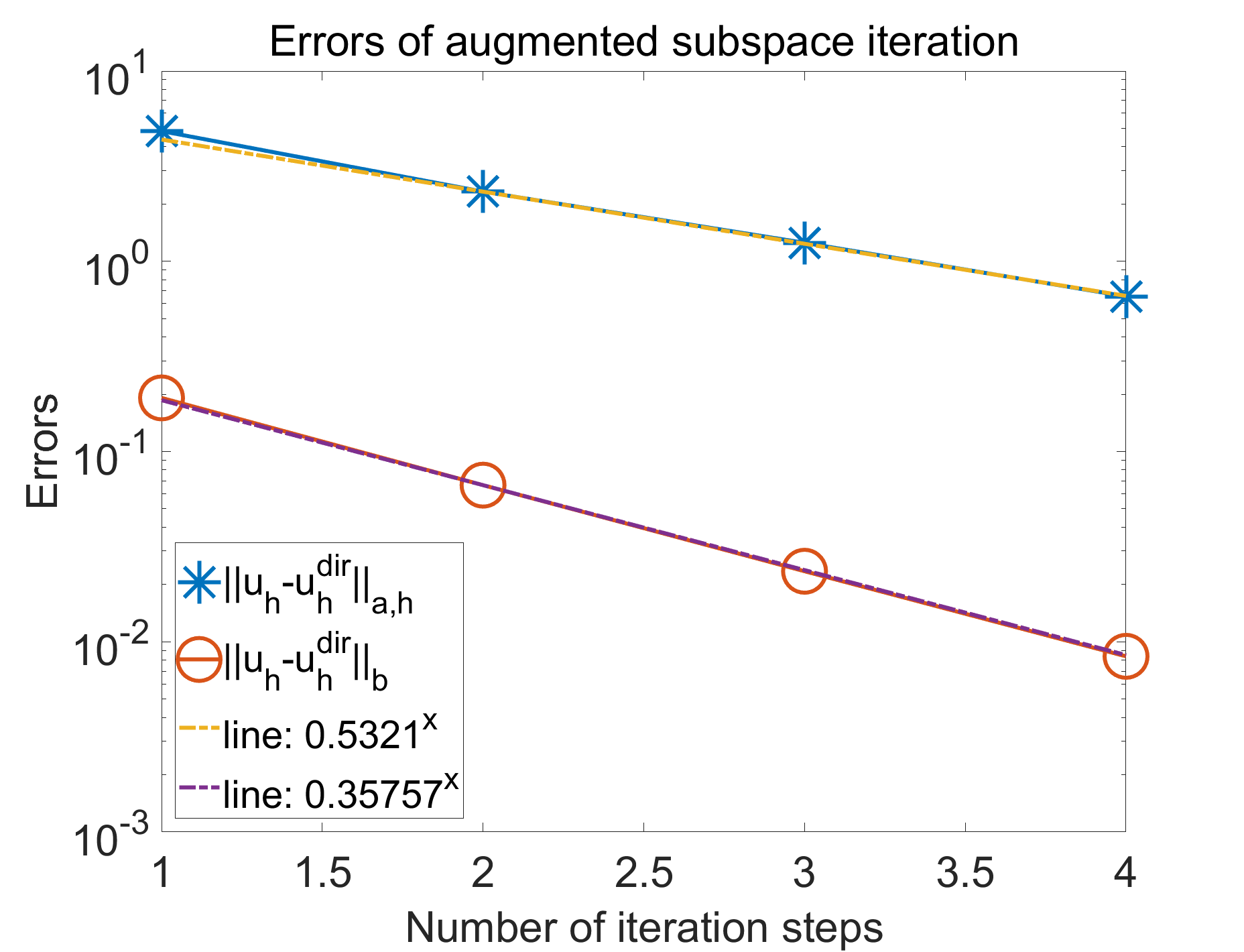}
\includegraphics[width=6cm,height=4.5cm]{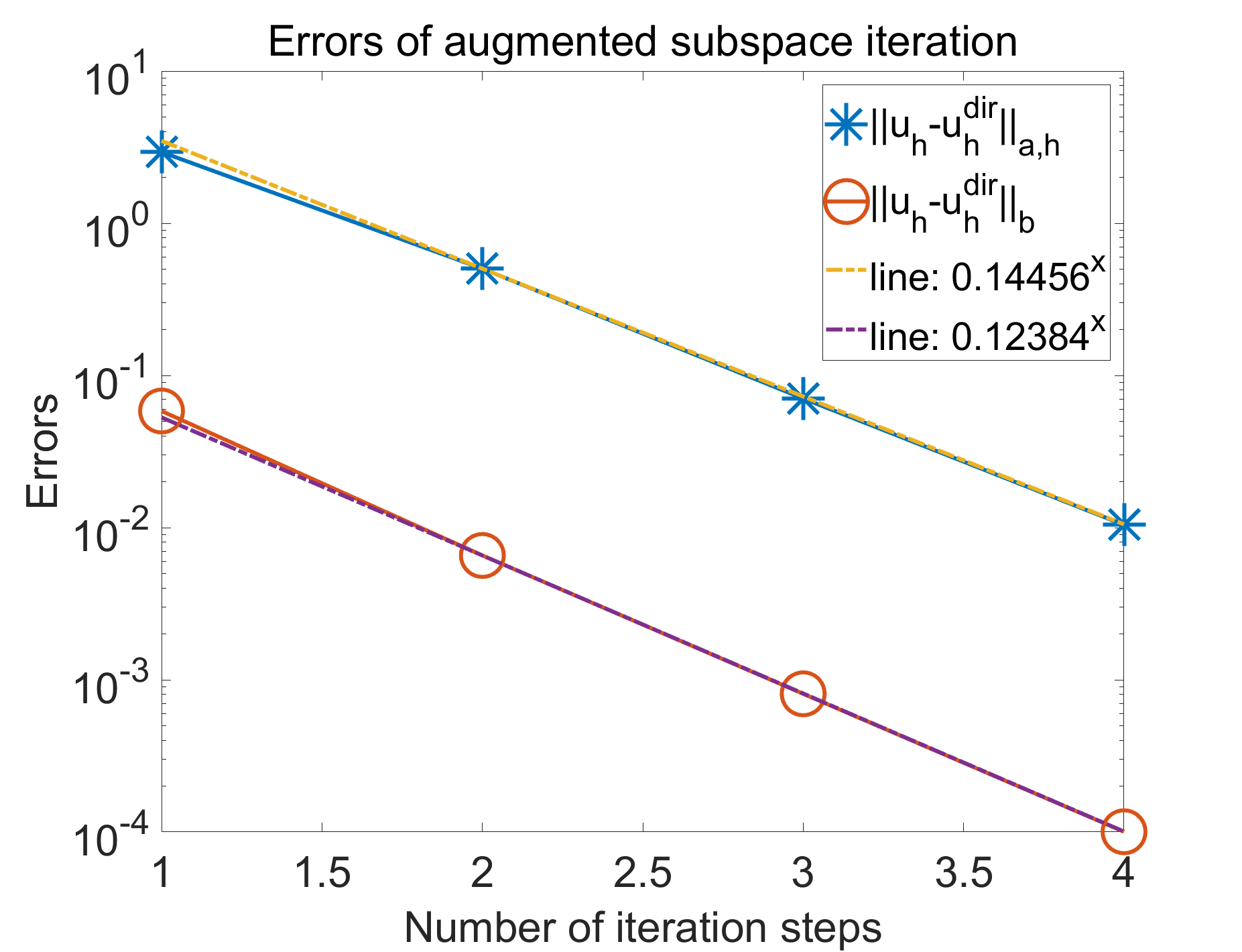}\\
\includegraphics[width=6cm,height=4.5cm]{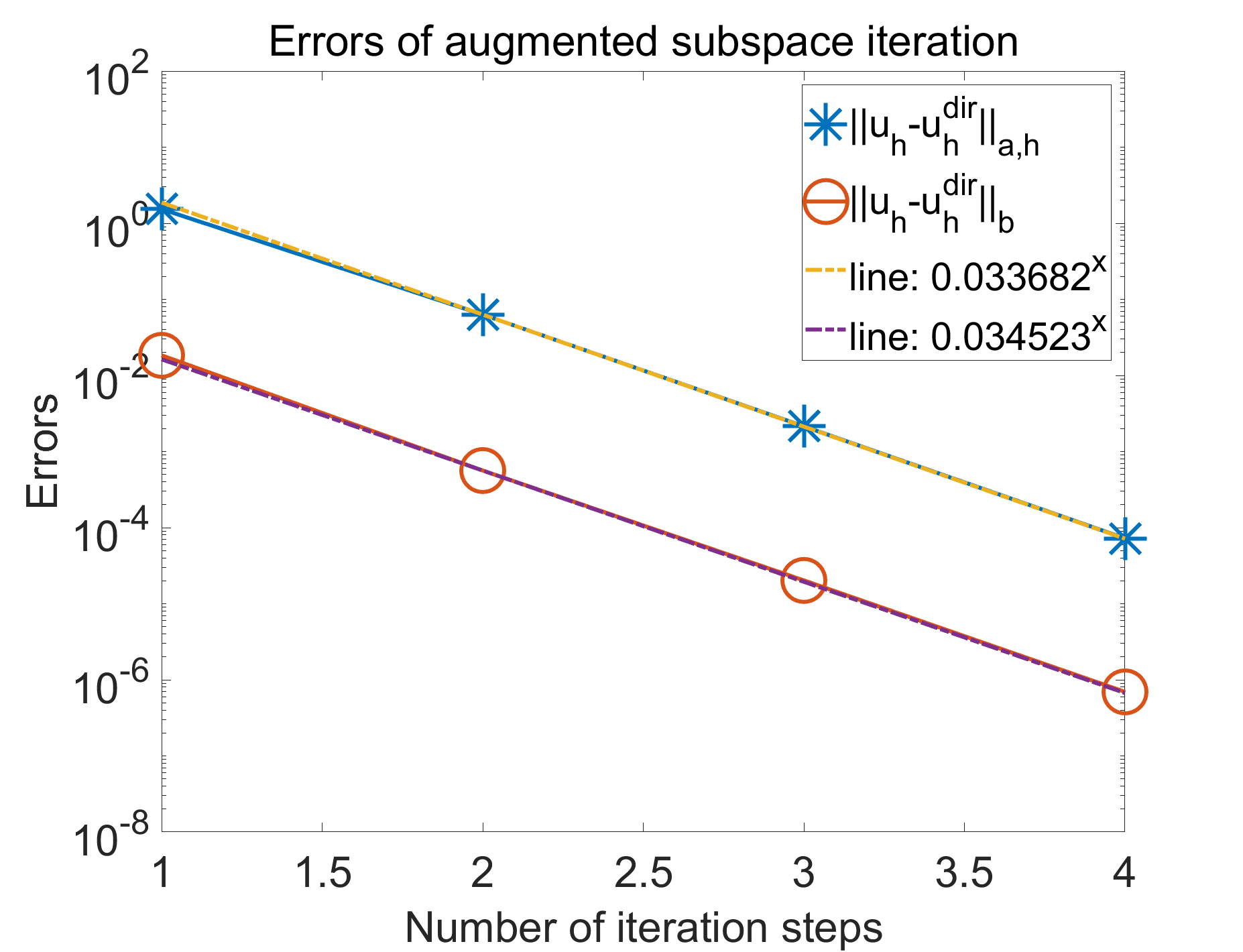}
\includegraphics[width=6cm,height=4.5cm]{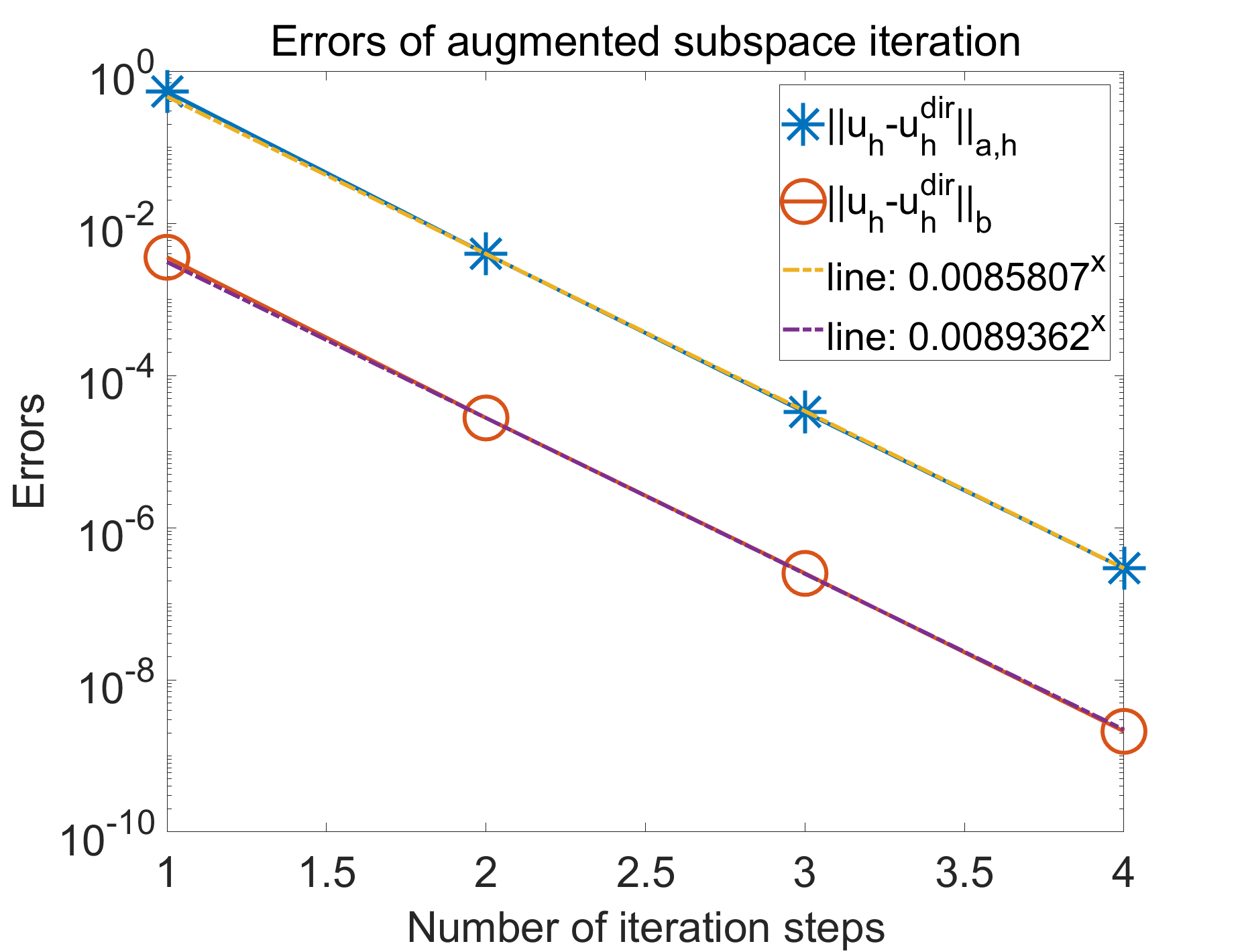}
\caption{The convergence behaviors for the only $4$-th eigenfunction by Algorithm \ref{Algorithm_1}
with the coarse space being the linear finite element space on the mesh with size $H=\sqrt{2}/8$, $\sqrt{2}/16$, $\sqrt{2}/32$ and $\sqrt{2}/64$, respectively.}\label{Result_Coarse_Mesh_4_Only}
\end{figure}

\section{Concluding remarks}
In this paper, some enhanced error estimates for the CR element based augmented subspace method are deduced for solving eigenvalue problems. Before the new estimates, the explicit error estimates for single eigenpair and multiple eigenpairs based on our defined spectral projection operators are derived, respectively. Then we prove the second order algebraic error convergence rate of the augmented subspace method. Based on the new algebraic error results, we can also
produce the corresponding sharper error estimates for the multigrid or multilevel methods
which are designed based on the augmented subspace method and the sequence of grids.

\section*{Acknowledgements}
This work was partly supported by the Beijing Natural Science Foundation (No. Z200003), the National Natural Science Foundation of China (No. 1233000214, 12301465), the National Center for Mathematics and Interdisciplinary Science, Chinese Academy of Sciences, and by the Research Foundation for Beijing University of Technology New Faculty (No. 006000514122516).

\section*{Declarations}

\subsection*{Conflict of interest}
All authors declare that they have no conflict of interest.

\subsection*{Data availability}
All data generated or analysed during the current study are available from the corresponding author on reasonable request.

%\bibliographystyle{elsarticle-harv}
%\bibliography{mybibfile}

\end{document}